\documentclass[preprint,1p]{elsarticle}

\makeatletter
\def\ps@pprintTitle{%
	\let\@oddhead\@empty
	\let\@evenhead\@empty
	\def\@oddfoot{\footnotesize\itshape
		{} \hfill}%
	\let\@evenfoot\@oddfoot
}
\makeatother
\usepackage[unicode]{hyperref}

\usepackage{latexsym}
\usepackage{indentfirst}
\usepackage{amsxtra}
\usepackage{amssymb}
\usepackage{amsthm}
\usepackage{amsmath}
\usepackage[dvipsnames]{xcolor}
\usepackage{mathrsfs} 

\usepackage[makeroom]{cancel}

\usepackage{soul}
\usepackage{tikz}

\usepackage[capitalise]{cleveref}

\newtheorem{theor}{Theorem}
\newtheorem{prop}[theor]{Proposition}
\newtheorem{lemma}[theor]{Lemma}
\newtheorem{cor}[theor]{Corollary}
\theoremstyle{definition}               
\newtheorem{defin}[theor]{Definition}
\newtheorem{ex}{Example}

\newtheorem{rem}[theor]{Remark}

\newtheorem{que}{Question}

\DeclareMathOperator{\Sym}{Sym}
\DeclareMathOperator{\Aut}{Aut}

\DeclareMathOperator{\id}{id}
\DeclareMathOperator{\Ker}{Ker}
\DeclareMathOperator{\im}{Im}

\DeclareMathOperator{\Soc}{Soc}
\DeclareMathOperator{\Zoc}{Zoc}
\DeclareMathOperator{\gr}{gr}
\DeclareMathOperator{\Z}{Z}
\DeclareMathOperator{\Ann}{Ann}

\begin{document}


\begin{frontmatter}
	\title{Nilpotency in left semi-braces
	}
	\author [unile] {Francesco~CATINO}
	\ead{francesco.catino@unisalento.it}
	\author [unibar] {Ferran~CED\'{O}}
	\ead{cedo@mat.uab.cat}
	\author [unile] {Paola~STEFANELLI}
	\ead{paola.stefanelli@unisalento.it}
	\address[unile]{Dipartimento di Matematica e Fisica ``Ennio De Giorgi"
	\\
	Università del Salento\\
	Via Provinciale Lecce-Arnesano \\
	73100 Lecce\\
	Italy
	\vspace{3mm}
	
	}\address[unibar]{Departament de Matem\`{a}tiques
	\\
	Universitat Aut\`{o}noma de Barcelona\\
	08193 Bellaterra (Barcelona) \\
	Spain\\
}
	
\begin{abstract}
	We introduce left and right series of left semi-braces. This allows to define left and right nilpotent left semi-braces. We study the structure of such semi-braces and generalize some results, known for skew left braces, to left semi-braces. We study the structure of left semi-braces $B$ such that the set of additive idempotents $E$ is an ideal of $B$. Finally we introduce the concept of a nilpotent left semi-brace and we show that the multiplicative group of such semi-braces is nilpotent.   
\end{abstract} 
	
\end{frontmatter}

\medskip

\section*{Introduction}

The quantum Yang-Baxter equation is one of the basic equations of statistical mechanics that takes its name from independent work of C.N. Yang \cite{Ya67} and R.J. Baxter \cite{Ba72}. Outside to a strictly physical perspective, this equation laid to the foundations of the theory of quantum groups and to the development of Hopf algebras theory, see \cite{Ka95}.\\
Determining and classifying all the solutions of the Yang-Baxter equation is still an open question. To attack this wide problem, Drinfel'd \cite{Dr92} suggested to study the special class of \emph{set-theoretical solutions}, namely, maps $r:X\times X\to X\times X$, with $X$ a set, satisfying the braid relation 
\begin{align*}
\left(r\times \id_X\right)
\left(\id_X\times r\right)
\left(r\times \id_X\right)
=
\left(\id_X\times r\right)\left(r\times \id_X\right)
\left(\id_X\times r\right).
\end{align*}
The papers of Etingof, Schedler and Soloviev \cite{ESS99}, Gateva-Ivanova and Van den Bergh \cite{GaVa98}, Lu, Yan and Zhu \cite{LuYZ00}, and Soloviev \cite{So00}, were the first to trace the path for studying set-theoretical solutions mainly in group theory terms and also by introducing some algebraic structures. In particular, they focused on bijective solutions which are non-degenerate. We recall that, writing $r\left(x,y\right)=\left(\lambda_x\left(y\right), \rho_y\left(x\right)\right)$, for all $x,y\in X$, with $\lambda_x,\rho_y$ maps from $X$ into itself, $r$ is said to be  \emph{left non-degenerate} if $\lambda_x\in\Sym_X$, for every $x\in X$ and, similarly, \emph{right non-degenerate} if  $\rho_y\in\Sym_X$, for every $y\in X$. Moreover, $r$ is \emph{non-degenerate} if it is both left and right non-degenerate.  

In 2007, a surprising connection between rings and solutions has been found by Rump \cite{Ru07}. More precisely, every Jacobson radical ring determines a non-degenerate set-theoretic solution that is also involutive, i.e., $r^2=\id_{X\times X}$. Motivated by this special link, Rump \cite{Ru07} introduced \emph{braces}, a generalization of Jacobson radical rings, for determining  and investigating involutive non-degenerate solutions.
We highlight that some structural properties of braces have been studied exactly to understand the behaviour of special solutions. In particular, the notion of left and right nilpotency have been introduced by Rump \cite{Ru07} in the context of braces to study multipermutation solutions in relation to Gateva–Ivanova’s strong conjecture \cite[2.28(I)]{Ga04}. In this context, important results have been provided by Smoktunowicz \cite{Sm18, Sm18-2}. In the subsequent years, the brace theory has grown considerably, as one can see in several papers, such as \cite{CeJeOk14, CCoSt16, BaCeJeOk19, BaCeJeOk18, MeBaEs19, La20, CeJeOk21}.
For a fuller treatment we refer the reader to \cite{Ce18} along the references therein.

Later, suitable generalizations of braces turned out to be useful tools for obtaining new kind of solutions. 
In particular, Guarnieri and Vendramin \cite{GuVe17}, introduced skew braces to attack the problem of finding bijective non-degenerate solutions, non necessarily involutive. Moreover, as shown by  Smoktunowicz and Vendramin \cite[Theorem 4.13]{SmVe18} any finite skew brace $B$ produces a solution $r$ satisfying $r^{2n}=\id_{B\times B}$, for a certain positive integer $n$.  Recent advances concerning skew braces can be found in \cite{CCoSt19, Ba18, BaNeYa20, AcBo20, AcLuVe20, JeKuVAVe19, Ze19, Cr21, DeC19, JeKuVaVe20x}.\\
In addition, Catino, Colazzo, and Stefanelli \cite{CaCoSt17} introduced (left) semi-braces to study left non-degenerate solutions. A triple $\left(B,+,\circ\right)$ is a left semi-brace if $\left(B,+\right)$ is a left cancellative semigroup (not necessarily abelian), $\left(B, \circ\right)$ is a group, and the relation
\begin{align*}
	a\circ\left(b+c\right)
	= a\circ b + a\circ \left(a^- + c\right)
\end{align*}
is satisfied, for all $a,b,c\in B$, where $a^-$ denotes the inverse of $a$ in $\left(B, \circ\right)$. Moreover, if $B$ is a left semi-brace, then the map $r:B\times B\to B\times B$ defined by
\begin{align*}
r\left(a,b\right)
= \left(a\circ \left(a^- + b\right), \left(a^- + b\right)^-\circ b\right),
\end{align*}
for all $a,b\in B$, is a left non-degenerate solution which is not bijective in general. 
Moreover, the restriction $s:= r_{|_{E\times E}}$ of \ $r$ to the set of idempotents $E$ of $\left(B, +\right)$ is an idempotent solution, i.e., $s^2 = s$.
In 2019, Jespers and Van Antwerpen \cite{JeAr19} introduced a slight generalization of left semi-braces and proved that, under mild assumptions, other degenerate solutions can be obtained, some of which are such that $r^3 = r$. 
As shown by Catino, Colazzo, and Stefanelli \cite[Corollary 13]{CaCoSt20-2}, particular finite left semi-braces $B$, including that in \cite{JeAr19}, determine solutions $r$ for which $r^n = r$, where $n$ is an integer closely linked with the structure $B$.
We mention that other generalizations of the brace structure and their connection with solutions are lately investigated in \cite{CCoSt20x-2} and \cite{CMaSt20x-2}.
\smallskip

\noindent\textbf{Convention.} Hereinafter, for a left semi-brace we mean the structure introduced in \cite{CaCoSt17}, named left cancellative left semi-brace in \cite{JeAr19}.
\smallskip

The aim of this work is to explore more deeply some structural  aspects of the semi-brace. Specifically, we pay particular attention how left and right nilpotency - concepts suitable translated from the skew brace theory - influence the whole semi-brace structure.
The paper of Ced\'{o}, Smoktunowicz, and Vendramin \cite{CeSmVe19} mainly inspired the direction of our research.
Initially, in a left semi-brace $B$ we introduce the new operation $\cdot$ defined as
\begin{align*}
a\cdot b := \lambda_{a}\left(a^-\right) + a\circ b+\lambda_{b}\left(b^-\right),
\end{align*}
for $a,b\in B$, where $\lambda_a\left(a^-\right) = a\circ\left(a^- + a^-\right)$.
Even if at first sight it could appear an  
``unusual'' operation,  one can check that, in the case of a skew brace it becomes 
\begin{align*}
a\cdot b := - a + a\circ b - b
\end{align*}
which coincides with that introduced in \cite{CeSmVe19} and \cite{KoSmVe18}. 
Moreover, in the case of a Jacobson radical ring, the operation $\cdot $ is the usual multiplication of the ring itself. 
Although this operation is weaker with respect to rings - for instance, in general, it is not associative in a left semi-brace - we translate some notions of skew brace theory into the context of semi-braces.

As a first step, we focus on reformulating the definition of the ideal of a left semi-brace to makes it more usable. In particular, the proof of our characterization of ideals of a left semi-brace makes use of the operation $\cdot$.
Then, to investigate right nilpotent left semi-braces, we introduce the notion of \emph{right series} of a left semi-brace $B$
$$
B^{\left(1\right)} \supseteq B^{\left(2\right)} \supseteq \cdots \supseteq B^{\left(n\right)}\cdots
$$
by setting $B^{\left(1\right)}:= B$ and, for every positive integer $n$,  $B^{\left(n + 1\right)} = B^{\left(n\right)}\cdot B + E$, where  $B^{\left(n\right)}\cdot B$ denotes the subgroup of $\left(G, +\right)$  generated by elements of the form $x\cdot b$, with $x\in B^{\left(n\right)}$ and $b\in B$, where $G:= B+0$ is the well-known skew brace contained in $B$. Then, $B$ is said to be \emph{right nilpotent} if there exists a positive integer $n$ such that $B^{\left(n\right)} = E$.
Similarly, to study left nilpotent left semi-braces, we define the \emph{left series} of $B$ 
$$
B^{1} \supseteq B^{2} \supseteq \cdots \supseteq B^{n}\cdots 
$$
by setting $B^{1}:= B$ and, for every positive integer $n$, $B^{n + 1} = B\cdot B^{n} + E$. Thus, $B$ is said to be \emph{left nilpotent} if there exists a positive integer $n$ such that $B^{n} = E$.
One can observe that if $B$ is a skew brace then the previous definitions coincides with those provided in \cite{CeSmVe19}.  Furthermore, we demonstrate that the sets $B^{\left(n\right)}$ and $B^{n}$ are respectively ideals and left ideals of $B$, coherently to what proved in \cite{Ru07} for left braces and in \cite{CeSmVe19} for skew left braces. 

To study right nilpotency in the left semi-brace structure, it is significant to investigate in which cases the set of idempotents $E$ is an ideal of a left semi-brace $B$.  Indeed, it is clear that any right nilpotent left semi-brace $E$ as an ideal.  Note that, in general, $E$ is a left ideal and that $E$ is an ideal if and only if $E$ is a normal subgroup of the multiplicative group $\left(B,\circ\right)$. In light of this fact, we provide some useful characterizations of the left semi-brace $B$  such that $E$ is an ideal of $B$, mainly in terms of the operation $\cdot$. We obtain also that, given a left semi-brace $B$, then $E$ is an ideal of $B$ if and only if $B = \im\varphi\circ\ker\varphi$, where $\varphi$ is an idempotent endomorphism of $(B,\circ)$ which coincides with the map $\rho_0$ given by $\rho_0\left(b\right)= \left(b^-+0\right)^-$.
Considering left semi-braces having $E$ as an ideal, we introduce the notion of \emph{generalized socle} of $B$, i.e., $\Zoc\left(B\right):= \Soc\left(B\right) + E$, where $\Soc\left(B\right)$ is the socle of $B$, see \cite[Definition 21]{CaCoSt17}. 
We prove that $\Zoc\left(B\right)$ is an ideal of $B$ and that the \emph{generalized socle series}, i.e., the series given by $\Zoc_1\left(B\right):= \Zoc\left(B\right)$ and 
$\Zoc_n\left(B\right):= \Soc_n\left(B\right) + E$, for every positive integer $n$, is a special $z$-series, a generalization of the $s$-series already known for skew-braces, see \cite[Definition 2.11]{CeSmVe19}. 
In particular, we provide some characterizations of right nilpotent left semi-braces $B$ assuming that the additive group of $G=B+0$ is a nilpotent group. Specifically, if $\left(G,+\right)$ is a nilpotent group, then $B$ is right nilpotent if and only if $B$ admits a $z$-series. Moreover, if $E$ is an ideal of $B$, then $B$ admits a $z$-series if and only if $B = \Zoc_n\left(B\right)$, for a certain positive integer $n$.
Regarding the left nilpotency, the main result that we prove is that, given a left semi-brace with $\left(G,+\right)$ a nilpotent group and $\left(E, \circ\right)$ a nilpotent group contained in the centralizer of $G$ in $\left(B, \circ\right)$, then $B$ is left nilpotent if and only if $\left(B,\circ\right)$ is a nilpotent group. 

Following \cite{CeSmVe19} and \cite{Sm18}, we suitable define the sequence 
$$
B^{[1]} \supseteq B^{[2]} \supseteq \cdots \supseteq B^{[n]}\cdots 
$$
where $B^{[1]} := B$ and, for every positive integer $n$, $B^{[n]}$ is s the additive subgroup of $B$ generated by the elements belonging to the sets 
$B^{[i]}\cdot B^{[n+1-i]}$, with $1\leq i\leq n$. 
Similarly for skew braces, one can check that such sets $B^{[n]}$ are left ideals of $B$ and also some relations of the last series with left and right series of $B$.

Finally, following \cite{CCoSt19} we introduce the concept of annihilator of a left semi-brace. This allows to define nilpotent left semi-brace and we show that the multiplicative group of a such semi-brace is nilpotent. We also show that the nilpotent left semi-braces $B$ such that the set of idempotents $E$ is an ideal of $B$ are right nilpotent. If moreover $G=B+0$ is finite and $(G,+)$ is nilpotent, then $B$ is left nilpotent.  

\bigskip

\section{Basic results}
In this section, we recall some basic definitions and properties of left semi-braces. Moreover, we introduce a new operation $\cdot$ already provided in the context of skew braces in \cite{CeSmVe19} and \cite{KoSmVe18}. 
\medskip

\begin{defin}[Definition 1, \cite{CaCoSt17}]
	A set $B$ with two operations $+$ and $\circ$ is said to be a \emph{left semi-brace} if $\left(B, +\right)$ is a left cancellative semigroup,  $\left(B, \circ\right)$ is a group and
\begin{align*}
	a\circ\left(b + c\right) = a\circ b + a\circ \left(a^- + c\right)
\end{align*}
holds, for all $a,b,c\in B$, where $a^-$ is the inverse of $a$ with respect to $\circ$.
\end{defin}
\noindent As usual, we assume that the multiplication $\circ$ has higher precedence than the addition and denote by $0$ the identity of the group $\left(B,\circ\right)$.
In particular, it holds that $0$ is a left identity of the semigroup $\left(B,+\right)$ since $0$ is  an idempotent with respect to the sum.
\medskip

We recall that the additive structure of any left semi-brace $B$ is a right group and 
$B = G + E$, where $E$ is the set of idempotents of $\left(B,+\right)$ and $G$ is the subgroup $B +0$ of $(B,+)$.
In particular, we have that $G$ is a skew left brace  and $E$ is a left semi-brace such that $\left(E, +\right)$ is a right zero semigroup.
Note that any group $\left(B,\circ\right)$ determines a left semi-brace by setting $a + b = b$, for all $a,b\in B$. Such a semi-brace is called the \emph{trivial left semi-brace} on the group $\left(B, \circ\right)$ and it is clear that in this case $G=\{0\}$ and $B = E$.
Moreover, every element $b\in B$ can be written in a unique way as
\begin{align*}
b =g_b + e_b,
\end{align*} 
with $g_b\in G$ and $e_b\in E$; we call $g_b$ the \emph{group part} of $b$ and $e_b$ the \emph{idempotent part} of $b$.

\medskip

As in \cite[Proposition 3]{CaCoSt17} and \cite[Proposition 6]{CaCoSt17}, let us define the maps 
\begin{align*}
\lambda_a:B\to B, &\, b\mapsto a\circ\left(a^- + b\right),\\
 \rho_b:B\to B, &\, a\mapsto \left(a^- + b\right)^-\circ b,
\end{align*}
for all $a,b\in B$. In particular, we have that the map 
$\lambda: B\to \Aut(B,+), \,a\mapsto\lambda_a$ is a homomorphism from the group $(B, \circ)$ to the group of the automorphisms of $\left(B,+\right)$.
We recall that if $g\in G$ and $b\in B$, by \cite[Proposition 5]{CaCoSt17}, it holds 
$\lambda_g\left(b\right) = - g + g\circ b$, that is exactly the definition introduced in the context of skew left braces.
Moreover, the map 
$\rho: B\to B^B, \,b\mapsto\rho_b$ is a semigroup anti-homomorphism
from $(B, \circ)$  into the monoid of the maps from $B$ to itself.

\medskip

The maps $\lambda_a$ and $\rho_b$ satisfy particular properties with respect to the elements in $G$ and those in $E$ which are mainly collected in \cite[Proposition 7]{CaCoSt17}. We recall some of them in the proposition below. 
\begin{prop}\label{prop:Proposition7}
    Let $B$ be a left semi-brace.
    Then the following properties hold:
    \begin{enumerate}
        \item $\lambda_b\left(E\right) = E$, for every $b\in B$;
        \item $\rho_a\left(b\right)\in G$, for  all $a,b\in B$;
        \item $b\in G$ if and only if $\lambda_b\left(0\right)=0$;
        \item $b\in E$ if and only if $\rho_c\left(b^-\right) = 0$
        , for every $c\in B$.
    \end{enumerate}
\end{prop}
Let us note that, for every element $b\in B$, one can write the group part and the idempotent one of $b$ respectively as
\begin{align*}
    g_b = b + 0 = \rho_{0}\left(b^-\right)^-
    \quad\text{and}\quad
    e_b = - g_b + b =\lambda_b\left(0\right).
\end{align*}

Note that, by $1.$ in 
\cref{prop:Proposition7}, it holds
\begin{align*}
b =g_b + e_b = g_b\circ\lambda_{g_b^-}\left(e_b\right)\in G\circ E, 
\end{align*} 
for every $b\in B$. 
Effectively, $\left(B, \circ\right)$ is the matched product of $\left(G, \circ\right)$ and $\left(E, \circ\right)$, thus any element $b$ in $B$ can be written in a unique way as $b = g\circ e$ (or $b=e\circ g$), with $g\in G$ and $e\in E$.  For more details we refer the reader to \cite[Theorem 10]{CCoSt20}.\\
Furthermore, observe that any product $a\circ b$ can be written as
\begin{align*}
a\circ b 
= a\circ\left(0 + b\right)
= a + \lambda_a\left(b\right),
\end{align*}
for all $a,b\in B$.
\bigskip

Now, let us introduce the operation $\cdot$ for a left semi-brace.
\begin{defin}
	Let $B$ be a left semi-brace. We define the following operation $\cdot$ on $B$, 
	\begin{align}\label{def:puntino}
	a\cdot b:= \lambda_{a}\left(a^-\right) + a\circ b + \lambda_{b}\left(b^-\right),
	\end{align}
	for all $a,b\in B$.
\end{defin}
\medskip

\noindent Let us observe that 
if $a,b\in B$, then $a\circ b = g_a + e_a + \lambda_a\left(b\right) = g_a + \lambda_a\left(b\right)$, hence we get
\begin{align}\label{def:lambda-def2}
    \lambda_a\left(b\right) = - g_a + a\circ b,
\end{align} 
for all $a,b\in B$.
This implies that $\lambda_{a}\left(a^-\right) = -g_a + 0 = -g_a$, for every $a\in B$, thus one can write the operation $\cdot$ as
\begin{align}\label{def:puntino-1}
    a\cdot b
    = -g_a + a\circ b - g_b,
\end{align}
for all $a,b\in B$. In addition, 
\begin{align}\label{def:puntino-2}
a\cdot b= \lambda_{a}\left(b\right) + \lambda_{b}\left(b^-\right),
\end{align}
for all $a,b\in B$. 
Indeed, by  \eqref{def:lambda-def2} we have that
\begin{align*}
\lambda_{a}\left(a^-\right) + a\circ b  
= -g_a + a\circ a^- + a\circ b
= -g_a + a\circ b
= \lambda_{a}\left(b\right).
\end{align*}
Moreover, if $B$ is a skew left brace, then $\lambda_{a}\left(a^-\right) = - a $, hence
$a\cdot b = -a + a\circ b -b$,
i.e., the operation \eqref{def:puntino} coincides with the operation defined in \cite{CeSmVe19} and \cite{KoSmVe18}.
\bigskip

\noindent Note that if $B$ is a left semi-brace and $a\in B$, then
\begin{align*}
a\cdot 0 = 0
\quad \text{and}\quad
0\cdot a = 0.
\end{align*}
Indeed, by \cref{prop:Proposition7}-$1.$,
we have that $\lambda_a\left(0\right)\in E$, and so, by \eqref{def:puntino-2}, $a\cdot 0 = \lambda_a\left(0\right) + 
\lambda_{0}\left(0^-\right)
= \lambda_{0}\left(0^-\right)
 = 0$. Moreover, 
$0\cdot a = \lambda_0\left(a\right) + \lambda_a\left(a^-\right)
= a + \lambda_a\left(a^-\right)
= a\circ a^- = 0$.
\bigskip

\noindent Furthermore, the operation $\cdot$ features the following behaviour with respect to the sum and the multiplication.
\begin{prop}\label{prop:cdot-circ}
Let $B$ be a left semi-brace.  Then, the following properties are satisfied:
\begin{enumerate}
\item $a\cdot \left(b + c\right) = a\cdot b + b + a\cdot c +\lambda_{b}(b^-)$,
\item $\left(a\circ b\right)\cdot c = a\cdot\left(b\cdot c\right) + b\cdot c + a\cdot c$,
\end{enumerate}
for all $a,b,c\in B$.
\begin{proof}
Firstly, observe that, by \cref{prop:Proposition7}-$1.$,
$\lambda_a\left(a^-\right) + a = \lambda_a\left(0\right)\in E$,
for every $a\in B$.
\begin{enumerate}
	\item Let $a, b, c\in B$. We have that 
	\begin{align*}
	    a\cdot &\left(b + c\right) 
	    = -g_a + a\circ\left(b+c\right) - g_{b+c}&\mbox{(by \eqref{def:puntino-1})}\\
	    &= -g_a + a\circ b + \lambda_a\left(c\right) - g_c -g_b\\
	    &= -g_a + a\circ b - g_b + b 
	    - g_a + a\circ c - g_c -g_b
	    &\mbox{(since $-g_b + b = e_b\in E$)}\\
	    &=a\cdot b + b + a\cdot c + \lambda_b\left(b^-\right).
	\end{align*}
	
	\item Let $a,b,c\in B$. We have that
	\begin{align*}
		a\cdot\left(b\cdot c\right) + b\cdot c + a\cdot c
		&= \lambda_a\left(b\cdot c\right) + \lambda_{b\cdot c}(\left(b\cdot c\right)^-) + b\cdot c + 
		\lambda_a\left(c\right) + \lambda_c\left(c^-\right)\\
		&=\lambda_a\left(b\cdot c\right) +
		\lambda_a\left(c\right) + \lambda_c\left(c^-\right)\\
		&= \lambda_a\left(\lambda_b\left(c\right) + \lambda_c\left(c^-\right)\right) + \lambda_a\left(c\right) + \lambda_c\left(c^-\right)\\
		&= \lambda_a\lambda_b\left(c\right) + \lambda_a\lambda_c\left(c^-\right) + \lambda_a\left(c\right) + \lambda_c\left(c^-\right)\\
		&= \lambda_{a\circ b}\left(c\right) + \lambda_a\left(\lambda_c\left(c^-\right) + c\right) +  \lambda_c\left(c^-\right)\\
		&= \lambda_{a\circ b}\left(c\right) +  \lambda_c\left(c^-\right)\\
		&= \left(a\circ b\right)\cdot c.
	\end{align*}
\end{enumerate}
\end{proof}
\end{prop}

Clearly, if $B$ is a skew left brace, then the properties showed in \cref{prop:cdot-circ} coincides with equalities $(1.1)$ and $(1.2)$ in \cite{CeSmVe19}. 
\smallskip

\begin{prop}\label{prop:be0}
Let $B$ be a left semi-brace. 
Then, $a\cdot b\in G$, for all $a,b\in B$. 
Moreover, 
\begin{enumerate}
    \item $a\cdot e = 0$, for all $a\in B$ and $e\in E$;
    \item $a\cdot b = a\cdot g_b$, for all $a, b\in B$;
    \item $a\cdot b + b = \lambda_a\left(b\right) + e_b $, for all $a,b\in B$.
\end{enumerate}
\begin{proof}
Let $a, b\in B$. Since $\lambda_b\left(b^-\right) = -g_b\in G$, we easily obtain that 
$a\cdot b = \lambda_a\left(b\right) + \lambda_b\left(b^-\right)\in G$.\\
$1.$ \ For every $a\in B$ and $e\in E$, it holds that
\begin{align*}
    a\cdot e
    = \lambda_a\left(e\right) + \lambda_e\left(e^-\right)
    = \lambda_a\left(e\right) + 0
    = 0,
\end{align*}
since $\lambda_a\left(e\right)\in E$.\\
$2.$ \ 
If $a,b\in B$, then
\begin{align*}
    a\cdot b
    &= - g_a + a\circ b - g_b&\mbox{(by \eqref{def:puntino-1})}\\
    &= -g_a + a\circ g_b + \lambda_a\left(e_b\right)-g_b\\
    &= -g_a + a\circ g_b -g_b&\mbox{(since $\lambda_a\left(e_b\right)\in E$)}\\
    &= a\cdot g_b &\mbox{(by \eqref{def:puntino-1})}
\end{align*}
and so the claim follows.\\
$3.$ \ For $a,b\in B$, by \eqref{def:puntino-1} it holds that
\begin{align*}
   a\cdot b + b &= -g_a+a\circ b - g_b + b = \lambda_a\left(b\right) + e_b,
\end{align*}
which is our claim.
\end{proof}
\end{prop}
\medskip

From now on, if $X$ and $Y$ are subsets of a left semi-brace $B$, we denote by $X\cdot Y$  the subgroup of $\left(G, +\right)$ generated by elements of the form $x\cdot y$, with $x\in X$ and $y\in Y$, i.e., 
\begin{align*}
X\cdot Y = \gr\bigl\langle x\cdot y \ | \ x\in X, \ y\in Y \bigl\rangle_+.
\end{align*}

\smallskip

\section{Ideals of a left semi-brace}

The main aim of this section is providing a characterization of ideals which makes its definition more usable. In particular, the operation $\cdot$ turns out to be useful to this purpose. Besides, we introduce a suitable notion of left ideal for left semi-braces.
\medskip

Initially, let us recall the notion of ideal of a left semi-brace introduced in \cite{CaCoSt17} and a preliminary lemma.
\begin{defin}[Definition 17, \cite{CaCoSt17}]\label{def:17-CaCoSt17}
	If $B$ is a left semi-brace, a subsemigroup $I$ of $\left(B, +\right)$ is said to be an \emph{ideal} of $B$ if the following conditions hold:
	\begin{enumerate}
		\item $I$ is a normal subgroup of $\left(B, \circ\right)$;
		\item $I\cap G$ is a normal subgroup of $\left(G, +\right)$;
		\item $\rho_b\left(n\right)\in I$, for all $b\in B$ and $n\in I\cap G$;
		\item $\lambda_g\left(e\right)\in I$, for all $g\in G$ and $e\in I\cap E$.
	\end{enumerate}
\end{defin}
\smallskip

\begin{lemma}\label{le:subse-subg-B}
	Let $B$ be a left semi-brace and $I$ a subset of $B$ such that $I + 0\subseteq I$. Thus,
	\begin{enumerate}
	    \item if $x\in I$, then $g_x\in I$,
	    \item if $I\cap G$ is a subgroup of $\left(G,+\right)$ and $I$ is a subsemigroup of $\left(B,+\right)$, then $e_x\in I$.
	\end{enumerate}
	\begin{proof}
		$1.$ \ It is enough to note that
		\begin{align*}
		g_x  = x + 0 \in I.
		\end{align*}
		$2.$ \ By $1.$, since $I\cap G$ is a subgroup of $\left(G,+\right)$, we have that $-g_x\in I$. 
		So, since $I$ is subsemigroup of $\left(B,+\right)$,
		\begin{align*}
		e_x =  - g_x + x\in I.
		\end{align*}
	\end{proof}
\end{lemma}

\noindent Note that, in particular, any ideal $I$ of a left semi-brace satisfies the assumptions of the previous lemma, since clearly $0\in I$.
\smallskip

In the following theorem, we show how  to  reformulate the definition of ideal, using the operation $\cdot$.
\begin{theor}\label{th:ideal-semi-b}
	Let $B$ be a left semi-brace and $I$ a subset of $B$.  Then, $I$ is an ideal of $B$ if and only if
	\begin{enumerate}
		\item $I + 0\subseteq I$;
		\item $I\cap G$ is a normal subgroup of $\left(G, +\right)$; 
		\item $\lambda_g\left(I\right)\subseteq I$, for every $g\in G$;
		\item $I$ is a normal subgroup of $\left(B, \circ\right)$. 
	\end{enumerate}
	\begin{proof}
		At first suppose that $I$ is an ideal of $B$. Then, conditions $1.$, $2.$, and $4.$ are clearly satisfied. 
		Let $g\in G$ and $x\in I$.
		By \cref{le:subse-subg-B}-$2.$, we have that $e_x\in I$, thus, by \cref{def:17-CaCoSt17}-$4.$, it follows that $\lambda_{g}\left(e_x\right)\in I$.
		Moreover, since by \cref{le:subse-subg-B}-$1.$ $g_x\in I$, we obtain that
		\begin{align*}
		    \lambda_{g}\left(x\right) 
		    &= -g + g\circ\left(g_x + e_x\right)
		    = -g + g\circ g_x + \lambda_{g}\left(e_x\right)\\
		    &= -g + g\circ (g_x +g^-)+g+ \lambda_{g}\left(e_x\right)\\
		    &= -g+\rho_{g^-}(g^-_x)^-+g + \lambda_{g}\left(e_x\right)\in I,
		\end{align*}
		by \cref{def:17-CaCoSt17}-$2.$, $3.$ and $4$.
		
		Conversely, suppose that $1.$, $2.$, $3.$, and $4.$ are satisfied. 
		Initially, note that if $x\in I$, 
		by $1.$ and \cref{le:subse-subg-B}-$1.$, we get $g_x\in I$. Hence, if $x,y\in I$, by $3.$, we have that $\lambda_{g_x^-}\left(y\right)\in I$, and, by $4.$, we obtain 
		\begin{align*}
		x + y = g_x + y = g_x\circ \lambda_{g_x^-}\left(y\right)\in I,
		\end{align*}
		i.e., $\left(I, +\right)$ is a subsemigroup of $\left(B, +\right)$.   
		Moreover, observe that clearly $1.$, $2.$, and $4.$ in \cref{def:17-CaCoSt17} are satisfied. 
		Thus, we only need to show that $3.$ in \cref{def:17-CaCoSt17} holds.
		Let $n\in G\cap I$ and $b\in B$. 
		Observe that, writing $b = h\circ e$ with $h\in G$ and $e\in E$, by $2.$ in \cref{prop:cdot-circ}, we have that 
		\begin{align}\label{eq:prop-BI}
		b\cdot n 
		= \left(h\circ e\right)\cdot n
		= h\cdot\left(e\cdot n\right) + e\cdot n + h\cdot n\in I.
		\end{align}
		Indeed, by $3.$ and $2.$, it holds that
		$h\cdot n = \lambda_{h}\left(n\right) - n\in I$. Moreover,
		\begin{align*}
		e\cdot n = \lambda_{e}\left(n\right) - n
		&= e\circ n - n
		= \left(e\circ n\circ e^-\right)\circ e - n\\
		&= e\circ n\circ e^- + \lambda_{e\circ n\circ e^-}\left(e\right) - n\\
		&= e\circ n\circ e^- - n&\mbox{(since $\lambda_{e\circ n\circ e^-}\left(e\right)\in E$)},
		\end{align*}
		hence, by $4.$ and $2.$, $e\cdot n\in I$. 
		Clearly, since $e\cdot n\in I\cap G$, it also follows that   $h\cdot\left(e\cdot n\right)\in I$, and so $b\cdot n\in I$.
		Thus, 
		\begin{align*}
		   \left(\rho_b\left(n\right)\right)^-
		    &= b^-\circ n^- + \lambda_{b^-}\left(b\right)
		    = b^-\circ n^- - g_{b^-}
		    = g_{b^-} - g_{b^-} + b^-\circ n^- - g_{b^-}\\
		    &= g_{b^-} +\lambda_{b^-}\left(n^-\right) - g_{b^-}
		    = g_{b^-} +b^-\cdot n^- + n^- - g_{b^-}.
		\end{align*}
		Hence, since $b^-\cdot n^- + n^- \in G\cap I$, by $2.$ it follows that $\left(\rho_b\left(n\right)\right)^-\in I$. Therefore, by $4.$, we obtain that $3.$ in \cref{def:17-CaCoSt17} is satisfied.
	\end{proof}
\end{theor}
\medskip

In addition, ideals satisfy the following useful properties.
\begin{cor}\label{cor:propI}
	If $I$ is an ideal of a left semi-brace $B$, then $B\cdot I\subseteq I$ and $I\cdot B\subseteq I$.
	\begin{proof}
		We have that $B\cdot I\subseteq I$ by \cref{prop:be0} and 
			\eqref{eq:prop-BI} in the proof of \cref{th:ideal-semi-b}.\\
			Now, let $x\in I$ and $b\in B$. Thus,
			\begin{align*}
			x\cdot b 
			&=  x\cdot g_b
			= \lambda_x\left(g_b\right) - g_b
			= x\circ\left(x^- + g_b\right) - g_b
			= x\circ\left(g_b - g_b + x^- + g_b\right) - g_b\\
			&= x\circ g_b + \lambda_x\left(- g_b + x^- + g_b\right) - g_b\\ 
			&= g_b\circ\left(g_b^-\circ x\circ g_b\right) + \lambda_x\left(- g_b + x^- + g_b\right) - g_b\\
			&= g_b + \lambda_{g_b}\left(g_b^-\circ x\circ g_b\right) + \lambda_x\left(- g_b + x^- + g_b\right) + 0 - g_b.
			\end{align*}
			Note that, by $4.$ and $3.$ in \cref{th:ideal-semi-b}, we have that  $\lambda_{g_b}\left(g_b^-\circ x\circ g_b\right)\in I$, by $4.$ in \cref{th:ideal-semi-b}, $x^-\in I$ and, by $1.$, $2.$ and $4.$ in \cref{th:ideal-semi-b}, $\lambda_x\left(- g_b + x^- + g_b\right)=-g_x+x\circ\left(-g_b+x^-+g_b\right)\in I$. Therefore, by $1.$ and $2.$ in \cref{th:ideal-semi-b}, $x\cdot b\in I$.
	\end{proof}
\end{cor}
\smallskip

In general, the converse of \cref{cor:propI} does not hold. Indeed,
let $(B,\circ)$ be a non-abelian simple group and $B$ the trivial left semi-brace on such a group.
Consider the subgroup $I=\gr \langle a\rangle_{\circ}$ of $(B,\circ)$ generated by $a\in B\setminus\{ 0\}$. Note that $x\cdot y = 0$, for all $x,y\in B$. Hence
$B\cdot I\subseteq I$ and $I\cdot B\subseteq I$, but $I$ is not a normal subgroup of $(B,\circ)$. 
\medskip

\begin{prop}\label{pr:ideal-semi-b}
	Let $B$ be a left semi-brace and $I$ a subsemigroup of $(B,+)$.  
Then, $I$ is an ideal of $B$ if and only if
\begin{enumerate}
	\item $I\cap G$ is a normal subgroup of $\left(G, +\right)$; 
	\item $\lambda_g(e)\in I$, for every $g\in G$ and $e\in I\cap E$;
	\item $I\cdot B\subseteq I$ and $B\cdot I\subseteq I$;
	\item $\lambda_{a^-\circ x\circ a}(0)\in I$, for every $x\in I$ and $a\in G\cup E$.
	\item $I$ is a subgroup of $(B,\circ)$.
\end{enumerate}
\end{prop}
\begin{proof}
	Suppose first that $I$ is an ideal of $B$. Clearly 1., 2., 5. are satisfied and by \cref{cor:propI}, 3. is also satisfied.
	By \cref{th:ideal-semi-b}, for every $x\in I$ and $a\in B$, $a^-\circ x\circ a\in I$. Hence
	$$\lambda_{a^-\circ x\circ a}(0)=-(a^-\circ x\circ a+0)+a^-\circ x\circ a\in I.$$
	Conversely, suppose that $1.$, $2.$, $3.$, $4.$ and $5.$ are satisfied.
	Note that $\lambda_g(x)=g\cdot g_x+g_x+\lambda_g(e_x)\in I$, for every $g\in G$ and $x\in I$.
	By \cref{th:ideal-semi-b}, it is enough to prove that $I$ is a normal subgroup of $(B,\circ)$.
	Let $x\in I$, $g\in G$ and $e\in E$. We have that
	\begin{align*}
	e^-\circ x\circ e &=\lambda_{e^-}(x+\lambda_x(e))\\
	&=\lambda_{e^-}(x)+\lambda_{e^-\circ x}(e)\\
	&=e^-\cdot x+x+\lambda_{e^-\circ x}(e)\\
	&=e^-\cdot x+x+\lambda_{e^-\circ x\circ e}(0)\in I.
	\end{align*}
	and
	\begin{align*}
	g^-\circ x\circ g &=g^-\circ x\circ g+\lambda_{g^-\circ x\circ g}(0)\\
	&=g^-+\lambda_{g^-}(x)+\lambda_{g^-}(\lambda_x(g))+\lambda_{g^-\circ x\circ g}(0)\\
	&=g^-+\lambda_{g^-}(x)+\lambda_{g^-}(\lambda_x(g))+g^--g^-+\lambda_{g^-\circ x\circ g}(0)\\
	&=g^-+\lambda_{g^-}(x)+\lambda_{g^-}(\lambda_x(g)-g)-g^-+\lambda_{g^-\circ x\circ g}(0)\\
	&=g^-+\lambda_{g^-}(x)+\lambda_{g^-}(x\cdot g)-g^-+\lambda_{g^-\circ x\circ g}(0)\in I. 
	\end{align*}
	Let $a\in B$. We have that $a=g_a+e_a=g_a\circ \lambda_{g^{-}_a}(e_a)$ with $g_a\in G$ and $e_a\in E$.
	Therefore $a^-\circ x\circ a\in I$, for every $x\in I$, and the result follows.	
\end{proof}

A special ideal of a left semi-brace is the socle, introduced in \cite[Definition 21]{CaCoSt17}.\\
If $B$ is a left semi-brace, we call \emph{socle} of $B$ the  set
\begin{align*}
\Soc\left(B\right) :=
\{ a \ | \ a\in B,  \ \rho_a =  \rho_0, \  \lambda_a =  \lambda_0\}.
\end{align*}
By \cite[Proposition 24]{CaCoSt17}, the socle of $B$ can be written also as
\begin{align*}
\Soc\left(B\right) =
\{ a \ | \ a\in G,\, \forall \, b\in B \ a + b = a\circ b, \  - a + b + a =  b + 0 \}.
\end{align*}

\noindent In particular, $\Soc\left(B\right)$ is a sub-skew left brace of $G$ and it is an ideal of $B$, as shown in \cite[Proposition 22]{CaCoSt17}.\\
Note that, if $a\in \Soc\left(B\right)$ and $b\in B$, it holds that
\begin{align*}
a\cdot b 
= \lambda_a\left(b\right) + \lambda_b\left(b^-\right)
= b + \lambda_b\left(b^-\right)
= b\circ b^-
= 0.
\end{align*}
Therefore, if $B$ is a skew left brace, the socle of $B$ can be written also as
\begin{align}\label{eq:socle-skew}
\Soc\left(B\right) =
\{ a \ | \ a\in B,\, \forall \, b\in B \ \ a\cdot b = 0, \  b + a =  a + b \}.
\end{align}
In fact, in this case, if $a,b\in B$ are such that $a\cdot b = 0$, we obtain that
\begin{align*}
\lambda_a\left(b\right)
= \lambda_a\left(b\right) - b + b
= a\cdot b + b
= 0 + b
= b.
\end{align*}
\medskip

Now, we introduce the definition of left ideal of a left semi-brace.
\begin{defin}\label{def:left-ideal}
	Let $B$ be a left semi-brace and $I$ a subset of $B$.  
	Then, $I$ is a \emph{left ideal} of $B$ if
	\begin{enumerate}
		\item $I + 0\subseteq I$;
		\item $I\cap G$ is a  subgroup of $\left(G, +\right)$; 
		\item $\lambda_g\left(I\right)\subseteq I$, for every $g\in G$;
		\item $I$ is a subgroup of $\left(B, \circ\right)$.
	\end{enumerate}
\end{defin}

\noindent 
Clearly, the set $E$ of idempotents of a left semi-brace $B$ is a left ideal of $B$.\\
Analogously to what proved for ideals,  by conditions $1.$, $3.$, and $4.$ the additive structure $\left(I, +\right)$ of any left ideal $I$ is in particular a subsemigroup of $\left(B, +\right)$,
since
\begin{align*}
    x + y 
    = g_x + y = g_x\circ\lambda_{g_x^-}\left(y\right),
\end{align*}
for all $x,y\in I$. Moreover, note that any ideal of $B$ is clearly a left ideal of $B$.

\smallskip

\section{Left semi-braces having the set of idempotents as an ideal}

In this section, we focus on the special class of left semi-braces for which the set of the idempotents is an ideal. Specifically, we provide some useful characterizations of left semi-braces $B$ such that $E$ is an ideal of $B$ both in terms of the operation $\cdot$ and the sum and the multiplication.
\smallskip

\begin{rem}\label{rem:E-ideal}
	It is clear that $E$ is an ideal of a left semi-brace $B$ if and only if $E$ is a normal subgroup of $\left(B, \circ\right)$. Since, in general, the multiplicative structure of $B$ is the matched of  product of $\left(G, \circ\right)$ and  $\left(E,\circ\right)$, if $E$ is an ideal of $B$ we have that the group $\left(B, \circ\right)$ is precisely the semidirect product of $\left(G, \circ\right)$ acting on $\left(E,\circ\right)$.
\end{rem}
\smallskip

The following is an example of left semi-brace contained in \cite[Example 2]{CaCoSt17} for which $E$ is an ideal. 
\begin{ex}\label{ex:phi}
	Let $\left(B,\circ\right)$ be a group, $\varphi$ an idempotent endomorphism of such a group and consider $a + b:= b\circ \varphi\left(a\right)$. Then, $B$ is a left semi-brace with $G=\im \varphi$ and $E = \Ker\varphi$. Hence, $E$ is clearly an ideal of $B$. 
	In addition, let us observe that, if $g,h\in G$, since $g = \varphi\left(t\right)$ with $t\in G$, we have that
	$$
	g + h 
	= h\circ \varphi\left(g\right)
	= h\circ \varphi^2\left(t\right)
	= h\circ \varphi\left(t\right)
	= h\circ g,
	$$
	namely the skew left brace $G$ is as in \cite[Example 1.3]{GuVe17}.
\end{ex}
\smallskip

The following example ensures that not every left semi-brace has the set of idempotents as an ideal.
\begin{ex}
	Let $S$ be the skew left brace where $a + b = a\circ b$, for all $a,b\in S$, where $\left(S, \circ\right)$ is the symmetric group $\Sym_3$. Let $T$ be the trivial left semi-brace with $\left(T, \circ\right)$ isomorphic to the cyclic group $C_2$ of two elements. We write $T=\{0,t\}$, where $0$ is the neutral element of $\left(T,\circ\right)$. Let $\alpha:T\to \Aut\left(S\right)$ be the homomorphism from the group $\left(T, \circ\right)$ into the automorphism group of $\left(S, \circ\right)$ such that $\alpha\left(t\right) = \iota$, where $\iota\left(a\right) = \left( 2\, 3\right)\circ a\circ\left( 2\, 3\right)$, for every $a\in S$. Thus, if $B$ is the left semi-brace given by the semidirect product of $S$ and $T$ via $\alpha$ (see \cite[Corollary 15]{CaCoSt17}), we have that $G = \{\left(a, 0\right) \, | \, a\in S\}$ and $E = \{\left(0, u\right) \, | \, u\in T\}$ and $E$ is not an ideal of $B$. In fact, if we consider the elements $\left(\left(1\,2\,3\right),\, t\right)$ and $\left(0,\, t\right)\in E$ we obtain that
	$\left(\left(1\,2\,3\right),\, t\right) \left(0,\, t\right) \left(\left(1\,2\,3\right),\, t\right)^- =
	\left(\left(1\,3\,2\right),\, t\right)\notin E$.
\end{ex}
\medskip

In the following proposition we provide a characterization of the left semi-braces $B$, such that $E$ is an ideal of $B$, in terms of the operation $\cdot$.
\begin{prop}\label{prop:E-id}
	Let $B$ be a left semi-brace. Then, $E$ is an ideal of $B$ if and only if $e\cdot b = 0$, for all $e\in E$ and $b\in B$. 
	\begin{proof}
		Let $e\in E$ and $g\in G$. Then 
		\begin{align*}
		g^-\circ e\circ g&= g^-\circ \lambda_e(g+0)\\
		&=g^-\circ (\lambda_e(g)+\lambda_e(0))\\
		&=g^-\circ (e\cdot g+g+\lambda_e(0))\\
		&= g^- + \lambda_{g^-}\left(e\cdot g\right) - g^- + \lambda_{g^-}\left(e\right)
		\end{align*}
		where $g^- + \lambda_{g^-}\left(e\cdot g\right) - g^-\in G$ and $\lambda_{g^-}\left(e\right)\in E$.
		It follows that 
		\begin{align*}
		g^-\circ e\circ g\in E
		&\iff g^- + \lambda_{g^-}\left(e\cdot g\right) - g^- = 0\\
		&\iff \lambda_{g^-}\left(e\cdot g\right) = 0
		&\mbox{($\lambda_{g^-}\left(e\cdot g\right)\in G$)}\\
		&\iff e\cdot g = \lambda_{g}\left(0\right)&\mbox{($\lambda_g^{-1} = \lambda_{g^-}$)}\\
		&\iff e\cdot g = 0. &\mbox{($g\in G$)}
		\end{align*}
		Since every $b\in B$ can be written as $b = g_b\circ\lambda_{g_b^-}\left(e_b\right)$ with $\lambda_{g_b^-}\left(e_b\right)\in E$, by \cref{rem:E-ideal} and \cref{prop:be0}-$2.$, the assertion is verified.
	\end{proof}
\end{prop}
\medskip

As a consequence of the previous proposition, we  provide some characterizations in terms of a special relation with respect to the elements of $G$.
\smallskip

\begin{cor}\label{cor:E-id-1}
	Let $B$ be a left semi-brace. Then, $E$ is an ideal of $B$ if and only if $a\cdot b = g_a\cdot g_b$, for all  $a,b\in B$.
	\begin{proof}
		At first let us assume that $E$ is an ideal of $B$. Thus, 
		\begin{align*}
		a\cdot b 
		&= \left(g_a\circ \lambda_{g_a^-}\left(e_a\right)\right)\cdot g_b&\mbox{by \cref{prop:be0}-$2.$}\\
		&= g_a\cdot\left(\lambda_{g_a^-}\left(e_a\right)\cdot g_b\right)
		+ \lambda_{g_a^-}\left(e_a\right)\cdot g_b
		+ g_a\cdot g_b&\mbox{by  \cref{prop:cdot-circ}-$2.$}\\
		&=g_a\cdot 0
		+ 0 + g_a\cdot g_b &\mbox{by \cref{prop:E-id} \ }\\
		&= 0 + g_a\cdot g_b &\mbox{by \cref{prop:be0}-$1.$}\\
		&= g_a\cdot g_b.
		\end{align*}
		Conversely, assume that $a\cdot b = g_a\cdot g_b$, for all $a,b\in B$. In particular, if $e\in E$ and $b\in B$ we have that
		$e\cdot b = e\cdot g_b = 0\cdot g_b = 0$.
		Therefore, by \cref{prop:E-id}, $E$ is an ideal of $B$.
	\end{proof}
\end{cor}

\smallskip

\begin{cor}\label{cor:E-id}
	Let $B$ be a left semi-brace. Then, $E$ is an ideal of $B$ if and only if $e\circ g = g + e$, for all $e\in E$ and $g\in G$.
	\begin{proof}
		Observe that, if $e\in E$ and $g\in G$, 
		\begin{align*}
		g + e
		= e\circ  e^-\circ\left(g + e\right) 
		= e\circ\left(e^-\circ g + \lambda_{e^-}\left(e\right)\right)
		= e\circ\left(e^-\circ g + 0\right)
		= e\circ\left(e^-\cdot g + g\right).
		\end{align*}
		Therefore, by \cref{prop:E-id} and \cref{prop:be0}-$2.$,  the claim follows.		
	\end{proof}
\end{cor}

\medskip

By the previous corollary we show the following result that is also related to what is stated in \cref{rem:E-ideal}.
\begin{theor}\label{theor:semidirect-product}
	Let $B$ be a left semi-brace. Then $E$ is an ideal of $B$ if and only if $B$ is isomorphic to a semidirect product $T\rtimes A$ of a trivial left semi-brace $T$ by a skew left brace $A$.
	\begin{proof}  Let $B$ be a left semi-brace such that $E$ is an ideal of $B$. We shall see that $B\cong E\rtimes G$, the semidirect product of $E$ by $G$ by the action $\lambda$ of $G$ over $E$ (see \cite[Corollary 15]{CaCoSt17}). Note that by \cref{cor:E-id},
		$$g\circ e\circ g^{-}=g\circ (g^{-}+e)=\lambda_g(e),$$
		for all $e\in E$ and $g\in G$. Hence $\lambda_g\in \Aut(E,+,\circ)$. Let $f\colon B\rightarrow E\rtimes G$ be the map defined by $f(a)=(e_a,g_a)$, for all $a\in B$. Then we have 
		$$f(a+b)=(e_b,g_a+g_b)=(e_a,g_a)+(e_b,g_b)=f(a)+f(b)$$
		and by \cref{cor:E-id},
		\begin{align*}f(a\circ b)&=f((g_a+e_a)\circ (g_b+e_b))=f(e_a\circ g_a\circ e_b\circ g_b)\\
		&=f(e_a\circ\lambda_{g_a}(e_b)\circ g_a\circ g_b)=
		f(g_a\circ g_b+e_a\circ\lambda_{g_a}(e_b))\\
		&=(e_a\circ\lambda_{g_a}(e_b),g_a\circ g_b)=(e_a, g_a)\circ (e_b,g_b)=f(a)\circ f(b),
		\end{align*} 
		for all $a,b\in B$. Hence $f$ is an isomorphism.
		Conversely, suppose that $B\cong T\rtimes A$, a semidirect product of a trivial left semi-brace $T$ and a skew left brace $A$ via a homomorphism $\alpha\colon (A,\circ)\rightarrow \Aut(T,+,\circ)$. The set of idempotents of $(T\rtimes A,+)$ is $T\times \{ 0\}$. By the definition of the semidirect product of left semi-braces, $T\times \{ 0\}$ is an ideal of $T\rtimes A$, hence $E$ is an ideal of $B$ and the result follows.   
	\end{proof}
\end{theor}
\medskip

It is known that if $B$ is a group for which there exists an idempotent endomorphism $\varphi$, then $B = \Ker\varphi \im\varphi$ and $\Ker\varphi \cap \im\varphi = \{1\}$. Vice versa, if there exist $N, H$ subgroups of $B$ such that $N$ is normal, $B= NH$ and $N\cap H = \{1\}$, then every element $b$ can be written in a unique way as $b=n_b h_b$ with $n_b\in N$ and $h_b\in H$ and the projection on $H$, i.e., the map $\varphi:B\to B$ given by $\varphi\left(b\right)= h_b$, for all $b\in B$, is an idempotent endomorphism of the group $B$. In the context of left semi-brace $B$ having $E$ as an ideal, one can show that such a projection is exactly the map $\rho_0$. Consequently, we have the following characterization.
\begin{theor}\label{th:E-ideal-varphi}
	Let $B$ be a left semi-brace. Then, $E$ is an ideal of $B$ if and only if there exists an idempotent endomorphism $\varphi$ of the group $\left(B,\circ\right)$ such that $G=\im\varphi$ and $E= \Ker\varphi$. Furthermore, in this case $\varphi=\rho_0$.
	\begin{proof}
		Initially, let us assume that $E$ is an ideal of $B$. 
		If $b\in B$, by \cref{cor:E-id}, we have that 
		$b^- 
		= \left(g_b + e_b\right)^- 
		= \left(e_b\circ g_b\right)^-
		= g_b^-\circ e_b^- 
		= g_b^- + \lambda_{g_b^-}\left(e_b^-\right)$.
		Since $\lambda_{g_b^-}\left(e_b^-\right)\in E$, it follows that $g_{b^-} =g_b^-$. Thus,  
		\begin{align*}
		\rho_0\left(b\right) 
		= \left(b^- + 0\right)^-\circ 0 
		= \left(g_{b^-}\right)^- 
		= g_b.
		\end{align*}
		Set $\varphi:=\rho_0$, it holds that $\varphi^2\left(b\right) = \varphi\left(g_b\right) = g_b = \varphi\left(b\right)$. In addition, if $a,b\in B$, by the proof of \cref{theor:semidirect-product}, we have that $\varphi(a\circ b)=g_a\circ g_b=\varphi(a)\circ \varphi(b)$, i.e., $\varphi$ is an endomorphism of $\left(B,\circ\right)$.\\  
		Finally, the converse part of the statement is trivial.
	\end{proof}
\end{theor}
\medskip

Let us observe that, although left semi-braces $B$ as in \cref{ex:phi} has the shape of that in \cref{th:E-ideal-varphi}, not every left semi-brace having $E$ as ideal of $B$ are of the type in \cref{ex:phi}.
Effectively, we may consider as a simple example the direct product $B$ of a skew left brace $S$ where $a + b= a\circ b$, for all  $a,b\in S$, with $\left(S,\circ\right)$ a non-abelian group, and a trivial left semi-brace $T$. In this way, $G = \{\left(a, 0\right) \, | \, a\in S\}$ and $E = \{\left(0, u\right) \, | \, u\in T\}$  and obviously, in general, 
$\left(a, 0\right)+\left(b, 0\right)\neq \left(b, 0\right)\circ\left(a, 0\right)$.

\smallskip

\section{Right series of a left semi-brace}

This section is devoted to investigating right nilpotent left semi-braces. In more detail, we pose particular attention to left semi-braces for which $\left(G,+\right)$ is a nilpotent group and admitting a special series of ideals called $z$-series. In this context, we analyze the role played by an ideal strictly linked to the n-th socle of a left semi-brace.
\medskip

Initially, we introduce the notion of right series of a left semi-brace which we show to be a chain consisting of ideals.

\medskip

\begin{defin}
Let $B$ be a left semi-brace. We define $B^{\left(1\right)}:= B$ and, for every positive integer $n$,
\begin{align*}
B^{\left(n + 1\right)} 
= B^{\left(n\right)}\cdot B + E
\end{align*}
We call the series
$B^{\left(1\right)} \supseteq B^{\left(2\right)} \supseteq \cdots \supseteq B^{\left(n\right)}\cdots $\, the \emph{right series} of $B$.
\end{defin}

Let us note that if $B$ is a skew left brace then the previous definition coincides with that provided in \cite[p. 1372]{CeSmVe19}. 
Moreover, similarly to \cite[Proposition 2.1]{CeSmVe19}, the sets $B^{\left(n\right)}$ are ideals of a left semi-brace $B$, as shown in the following proposition.
\begin{prop}\label{prop:B(n)-ideal}
	Let $B$ be a left semi-brace. Then $B^{\left(n\right)}$ is an ideal of $B$, for every positive integer $n$.
	\begin{proof}
		We show the claim by using \cref{th:ideal-semi-b} and proceeding by induction on $n$.  The case $n=1$ is trivially satisfied. We assume that the assertion is true for some $n\geq 1$. 
		Initially, let us note that clearly  $B^{\left(n+1\right)} + 0\subseteq B^{\left(n+1\right)}$, so the condition $1.$ in \cref{th:ideal-semi-b} holds.
		To prove that condition $2.$ in \cref{th:ideal-semi-b} is satisfied, i.e., $B^{\left(n+1\right)}\cap G$ is a normal subgroup of $(G,+)$, initially observe that, by \cref{prop:be0}, 
		$$B^{\left(n+1\right)}\cap G = B^{\left(n\right)}\cdot B
		=\gr\langle x\cdot b\mid x\in B^{\left(n\right)}\mbox{ and }b\in B\rangle_+.
		$$
		Now, we show that $\left(B^{\left(n+1\right)}\cap G, +\right)$ is a normal subgroup of $\left(G,+\right)$.
		Let $g\in G$, $x\in B^{\left(n\right)}$, and $b\in B$. Note that
		\begin{align*}
		x\cdot g+g+\lambda_x(-g)+0&=\lambda_x(g)-g+g+\lambda_x(-g)+0\\
		&=\lambda_x(g)+\lambda_x(-g)+0\\
		&=\lambda_x(0)+0=0.
	    \end{align*} Hence
        \begin{align}\label{eq:oposit}
        &g+\lambda_x(-g)+0=-x\cdot g.
        \end{align} 
        Then,
		\begin{align*}
		g + x\cdot b  - g 
		&= g + \lambda_x\left(b\right) - (b+0) - g\\      
		&= g + \lambda_x\left(-g+g+b\right) - \left(g + b+0\right)\\
		&= g + \lambda_x\left(- g\right) + \lambda_x\left(g + b\right) - \left(g + b+0\right)\\
		&= -x\cdot g + \lambda_x\left(g + b\right) - \left(g + b+0\right)& (\mbox{by }\cref{eq:oposit})\\
		&= -x\cdot g + x\cdot\left(g + b\right)\in B^{\left(n\right)}\cdot B.
		\end{align*}
		Thus, $2.$ in \cref{th:ideal-semi-b} is satisfied.		

Now, let us prove the condition $3.$ in \cref{th:ideal-semi-b}. Since $\lambda_g(e)\in E$, for every $e\in E$ and $g\in G$, it is enough to show that $\lambda_g(x\cdot b)\in B^{\left(n\right)}\cdot B$, for every $g\in G$, $x\in B^{\left(n\right)}$ and $b\in B$. Let $g\in G$, $x\in B^{\left(n\right)}$,  $b\in B$. We have 
\begin{align*}
\lambda_{g}\left(x\cdot b\right)&=\lambda_{g}\left(x\cdot g_b\right)&\mbox{(by \cref{prop:be0}-2.)}\\
&= \lambda_{g}\left(\lambda_x\left(g_b\right) - g_b\right)\\
&= \lambda_{g\circ x\circ g^-}\lambda_{g}\left(g_b\right) -\lambda_{g}\left( g_b\right)\\
&= \left(g\circ x\circ g^-\right)\cdot\lambda_{g}\left(g_b\right),
\end{align*}
and thus $\lambda_{g}\left(x\cdot b\right)\in B^{\left(n+1\right)}$ by the inductive hypothesis.

Finally, we prove that condition $4.$ in \cref{th:ideal-semi-b} holds. 
At first, observe that $\left(B^{\left(n+1\right)}, \circ\right) \leq\left(B,\circ\right)$.
Indeed, if $a,b\in B^{\left(n+1\right)}$,  
recalling that $x\circ y = x+ \lambda_x\left(y\right)$, for all $x,y\in B$, it follows that
\begin{align*}
a\circ b^-
&=a+\lambda_a(b^-)
= a + \lambda_a(b^-)\circ 0\\
&=a+\lambda_a(b^-) +\lambda_{\lambda_a(b^-)}(0)\\
&=a+a\cdot b^-+b^-+\lambda_{\lambda_a(b^-)}(0) &(\mbox{since }\lambda_{b^-}(b)+b^-\in E)\\
&=a+a\cdot b^-+\lambda_{g_b^-}(e_b)^-\circ g_b^-+\lambda_{\lambda_a(b^-)}(0)\\
&=a+a\cdot b^-+\lambda_{g_b^-}(e_b)^-\cdot g_b^-+g_b^-+\lambda_{\lambda_a(b^-)}(0)&(\mbox{since }\lambda_{g_b^-}(g_b)+g_b^-\in E)\\
&=a+a\cdot b^-+\lambda_{g_b^-}(e_b)^-\cdot g_b^--\lambda_{g_b^-}(g_b)+\lambda_{\lambda_a(b^-)}(0)\in B^{\left(n+1\right)},
\end{align*} 
by  condition $3.$ in \cref{th:ideal-semi-b} and since $B^{\left(n+1\right)}\subseteq B^{\left(n\right)}$. 
Hence $(B^{\left(n+1\right)},\circ)$ is a subgroup of $(B,\circ)$. 
We will prove that $\left(B^{\left(n+1\right)}, \circ\right)$ is normal in $\left(B,\circ\right)$.
Let $a\in B^{\left(n+1\right)}$ and $c\in B$. Note that
\begin{align*}
c\circ a\circ c^-&=g_c\circ\lambda_{g_c^-}(e_c)\circ a\circ\left(\lambda_{g_c^-}(e_c)\right)^-\circ g_c^-\\
&=g_c\circ\lambda_{g_c^-}(e_c)\circ g_c^-\circ g_c\circ a\circ g_c^-\circ g_c\circ\left(\lambda_{g_c^-}(e_c)\right)^-\circ g_c^-.
\end{align*}
Hence, in order to proof that $c\circ a\circ c^-\in B^{\left(n+1\right)}$, we may assume that $c\in G$. Now we have
\begin{align*}
c\circ a \circ c^-
&=c+\lambda_c(a+\lambda_a(c^-))\\
&=c+\lambda_c(a)+\lambda_{c\circ a}(c^-)\\
&= c + \lambda_c(a) + \lambda_{c\circ a}(c^-)\circ 0\\
&=c+\lambda_c(a)+\lambda_{c\circ a}(c^-)+\lambda_{\lambda_{c\circ a}(c^-)}(0)\\
&=c+\lambda_c(a)+\lambda_{c\circ a\circ c^-}(\lambda_{c}(c^-))+\lambda_{\lambda_{c\circ a}(c^-)}(0)\\
&=c+\lambda_c(a)+\lambda_{c\circ a\circ c^-}(-c)+c-c+\lambda_{\lambda_{c\circ a}(c^-)}(0)\\
&=c+\lambda_c(a)+(c\circ a\circ c^-)\cdot(-c)-c+\lambda_{\lambda_{c\circ a}(c^-)}(0).
\end{align*}  
Since $ B^{\left(n+1\right)}\subseteq  B^{\left(n\right)}$, by the induction hypothesis, $c\circ a\circ c^-\in  B^{\left(n\right)}$, and thus $(c\circ a\circ c^-)\cdot(-c)\in B^{\left(n+1\right)}$. By \cref{th:ideal-semi-b}-3., $\lambda_c(a)\in B^{\left(n+1\right)}$. Hence, by \cref{th:ideal-semi-b}-2.,
$$c+\lambda_c(a)+(c\circ a\circ c^-)\cdot(-c)-c\in B^{\left(n+1\right)}.$$
Therefore $c\circ a\circ c^-\in B^{\left(n+1\right)}$ and the result follows by induction.						
\end{proof}
\end{prop}
\bigskip 

Let us introduce the following definition.
\begin{defin}
	A left semi-brace $B$ is said to be \emph{right nilpotent} if $B^{\left(n\right)} = E$ for some positive integer $n$.
\end{defin}
\noindent Clearly, if $B$ is a skew left brace, then the previous definition coincides with Definition $2.4$ in \cite{CeSmVe19}.
Besides, note that if a left semi-brace $B$ is right nilpotent, then $E$ is trivially an ideal of $B$.
\medskip

Now, we introduce the notion of generalized socle. 
\begin{defin}\label{def:gen-zoc}
	Let $B$ be a left semi-brace such that $E$ is an ideal of $B$. Then, let us define the set
	\begin{align*}
	\Zoc\left(B\right) := \Soc\left(B\right) + E
	\end{align*}
	which we call \emph{generalized socle} of $B$.
\end{defin}

Note that, in general,  $\Zoc\left(B\right)\subseteq \Soc\left( G\right) + E$ but $\Zoc\left(B\right)\neq\Soc\left( G\right) + E$.
\begin{ex}\label{ex:semi-braceS3}
	For instance, let $B$ be the left semi-brace recalled in \cref{ex:phi}. Thus, as observed in \cite[Example 25]{CaCoSt17}, $\Soc\left(B\right) = \Z\left(B, \circ\right)\cap \im\varphi$. Moreover,  $\Soc\left(G\right) = \Z\left(G, \circ\right)$. Indeed, if $g\in \Soc\left(G\right)$ and $h\in G$, since $g\in \im\varphi$, there exists $a\in B$ such that $g = \varphi\left(a\right)$ and so
	$g\circ h  = g+h  = h\circ \varphi^2\left(a\right) = h\circ \varphi\left(a\right) = h\circ g$, hence $g\in \Z\left(G, \circ\right)$. 
	Conversely, if $g\in \Z\left(G, \circ\right)$ and $h\in G$, since $g=\varphi\left(a\right)$ and $h=\varphi\left(b\right)$ for certain $a,b\in B$, we obtain that
	$g + h  = h\circ \varphi^2\left(a\right) = h\circ \varphi\left(a\right) = h\circ g
	= g\circ h$ and also
	$h + g = g\circ \varphi^2\left(b\right)
	 = g\circ \varphi\left(b\right)
	  = g\circ h
	  = g+h$, therefore $g\in \Soc\left(G\right)$.\\
	Now, we choose as specific group $\left(B,\circ\right)$ the symmetric group $\Sym_3$ and $\varphi$ the idempotent endomorphism of $\Sym_3$ defined by setting $\varphi\left(12\right) = \left(12\right)$ and $\varphi\left(123\right) = \id$. In this case, $G = \im\varphi = \langle\left(12\right)\rangle$ and $E=\ker\varphi = \langle\left(123\right)\rangle$. Moreover, since $\Soc\left(B\right) = \Z\left(B, \circ\right)\cap \im\varphi$, we have that  $\Soc\left(B\right)$ is trivial and, consequently, $\Zoc\left(B\right) = \langle\left(123\right)\rangle$. 
	On the other hand,  $\Soc(G) = \Z(\langle (12)\rangle, \circ)= \langle\left(12\right)\rangle$. It follows that the element $\left(12\right)\in\Soc\left(G\right) + E$ does not belong to $\Zoc\left(B\right)$.
\end{ex}
\medskip

\begin{prop}
	Let $B$ be a left semi-brace such that $E$ is an ideal of $B$. Then, $\Zoc\left(B\right)$ is an ideal of $B$. 
	\begin{proof}
		We verify our assertion by using \cref{th:ideal-semi-b}.
		Let $a\in\Soc\left(B\right)$ and $e\in E$. Then, since $a\in G$, we have that $a + e + 0 = a + 0 = a\in\Zoc\left(B\right)$. 
		Moreover, $\Zoc\left(B\right)\cap G = \Soc\left(B\right)$ is a normal subgroup of $\left(G, +\right)$, since $\Soc\left(B\right)$ is an ideal of $B$. 
		Furthermore, if $g\in G$, 
		\begin{align*}
		\lambda_g\left(a + e\right) 
		= \lambda_g\left(a\right) + \lambda_g\left(e\right)
		\in \Soc\left(B\right) + E.
		\end{align*}
		Now, to prove the closure with respect to $\circ$, at first observe that if $e\in E$ and $b\in\Soc\left(B\right)$, then $e\circ b = e\circ b\circ e^-\circ e\in \Zoc\left(B\right)$, since $e\circ b\circ e^-\in\Soc\left(B\right)$. Hence, 
		if $a,b\in \Soc\left(B\right)$ and $e,f\in E$, since $a + e = a\circ e$ and $\lambda_a = \lambda_b = \id_B$, we get
		\begin{align*}
		\left(a + e\right)\circ \left(b + f\right)
		&= a + e + \lambda_{a + e}\left(b\right) + \lambda_{a + e}\left(f\right) 
		= a + \lambda_{a\circ e}\left(b\right) + \lambda_{a\circ e}\left(f\right)\\
		&= a + \lambda_{e}\left(b\right) + \lambda_{e}\left(f\right)
		= a + e\circ b + e\circ f\in \Zoc\left(B\right).
		\end{align*}
		In addition, $\left(a + e\right)^- = e^-\circ a^-\in \Zoc\left(B\right)$.
		Finally, if $b\in B$, 
		\begin{align*}
		b\circ \left(a + e\right)\circ b^-
		= b\circ a\circ e\circ b^-
		= b\circ a\circ b^-\circ b\circ e\circ b^-\in \Zoc\left(B\right),
		\end{align*}
		since $\Soc\left(B\right)$ and $E$ are normal in $\left(B, \circ\right)$.
		Therefore, by \cref{th:ideal-semi-b}, the claim follows.
	\end{proof}
\end{prop}

\medskip

Let us note that, under the assumption of $E$ ideal, considered the relations of congruence $\sim_{\Zoc\left(B\right)}$ and $\sim_{\Soc\left(B\right)}$ in the sense of \cite[Proposition 19]{CaCoSt17}, we obtain that
\begin{align*}
    a\sim_{\Zoc\left(B\right)} b
    \iff g_a\sim_{\Soc\left(B\right)} g_b,
\end{align*} 
for all $a,b\in B$.
Indeed, it is enough to observe that, if $a,b\in B$, then
\begin{align*}
    a\circ b^- 
    &= g_a + \lambda_{a}\left(g_{b^-}\right)
    +\lambda_{a}\left(e_{b^-}\right)\\
    &= g_a + \lambda_{e_a}\lambda_{g_a}\left(g_b^-\right)
    +\lambda_{a}\left(e_{b^-}\right)&\mbox{by \cref{cor:E-id} and $g_{b^-} = g_b^-$}\\
    &=  g_a + e_a\circ\lambda_{g_a}\left(g_b^-\right)
    +\lambda_{a}\left(e_{b^-}\right)\\
    &=  g_a + \lambda_{g_a}\left(g_b^-\right) + e_a
    + \lambda_{a}\left(e_{b^-}\right)&\mbox{by \cref{cor:E-id}}\\
    &= g_a\circ g_b^- + \lambda_{a}\left(e_{b^-}\right),
\end{align*} 
with $g_a\circ g_b^-\in G$ and $\lambda_{a}\left(e_{b^-}\right)\in E$.

\bigskip

To describe the generalized socle of a left semi-brace in the same terms of the definition of the socle, namely, in terms of the maps $\rho_a$ and $\lambda_a$, we firstly provide the following lemma.
In particular, it contains what is proved about the map $\rho_0$ in the proof of \cref{th:E-ideal-varphi}.

\begin{lemma}\label{lem:rhoe}
  Let $B$ be a left semi-brace such that $E$ is an ideal of $B$. Then, $\rho_{e}\left(b\right) = g_b$, for all $e\in E$ and $b\in B$.
   \begin{proof}
       As shown in the proof of \cref{th:E-ideal-varphi}, we have that $g_{b^-} =g_b^-$, for each $b\in B$. Then, if $e\in E$ and $b\in B$, by \cref{cor:E-id}
       it holds that
       \begin{align*}
        \rho_{e}\left(b\right) 
           = \left(b^- + e\right)^-\circ e
           = \left(g_{b^-} + e\right)^-\circ e
           = \left(e\circ g_{b}^-\right)^-\circ e
           = g_b,
       \end{align*}
       which is our claim. 
   \end{proof}
\end{lemma}

\begin{prop}
    Let $B$ be a left semi-brace such that $E$ is an ideal of $B$. Then,
    \begin{align*}
        \Zoc\left(B\right)
        = \{a \ | \ a\in B\,, \ \rho_a = \rho_0, \  \lambda_{a}
        = \lambda_{e_a}\}.
       \end{align*}
    \begin{proof}
    At first, set 
    $A:= \{a \ | \ a\in B\,, \ \rho_a = \rho_0, \  \lambda_{a}
        = \lambda_{e_a}\}$.\\
    Let $a\in \Zoc\left(B\right)$. Then, $g_a\in\Soc\left(B\right)$ and so it holds that $a = g_a + e_a = g_a\circ e_a$. 
    Since $\lambda_{g_a} = \id_B$, we obtain that
    \begin{align*}
        \lambda_a\left(b\right) 
        = \lambda_{g_a\circ e_a}\left(b\right)
        = \lambda_{g_a}\lambda_{e_a}\left(b\right)
        = \lambda_{e_a}\left(b\right)
        = e_a\circ b,
    \end{align*}
    for all $b\in B$.
    Moreover, since $\rho_{g_a} = \rho_{0}$, by \cref{lem:rhoe} it follows that 
    \begin{align*}
        \rho_a\left(b\right)
        = \rho_{e_a}\rho_{g_a}\left(b\right)
        = \rho_{e_a}\rho_{0}\left(b\right)
        = \rho_{e_a}\left(g_b\right)
        = g_{b}.
    \end{align*}
    Therefore, $a\in A$.
    
    Now, let $a\in A$. By \cref{cor:E-id}, $a=e_a\circ g_a$. Since $g_a^-=g_{a^-}$, we have that
    \begin{align*}
        g_a^- + b
        &= g_{a^-} + e_{a^-} + b
        = a^- + b
        = a^- \circ\lambda_{a}\left(b\right)
        = a^-\circ e_a\circ b
        = g_a^-\circ e_a^-\circ e_a\circ b = g_a^-\circ b,
    \end{align*}
    for all $b\in B$, which implies that $\lambda_{g_a} = \id_B = \lambda_0$. Furthermore, by \cref{lem:rhoe},
    \begin{align*}
        \rho_{g_a}\left(b\right)
        &= \rho_{e_a^- \circ e_a\circ g_a}\left(b\right)
        = \rho_{a}\rho_{e_a^-}\left(b\right)
        = \rho_{0}\left(g_b\right) = g_b= \rho_{0}\left(b\right).
    \end{align*}
    Hence, $g_a\in \Soc\left(B\right)$, consequently $a\in \Zoc\left(B\right)$. 
    Therefore, the claim follows.
    \end{proof}
\end{prop}

\begin{cor}
    Let $B$ be a left semi-brace such that $E$ is an ideal of $B$. Then,
   \begin{align*}
        \Zoc\left(B\right)
        = \{a \ | \ a\in B\,, \ \rho_a = \rho_{e_a}, \  \lambda_{a}
        = \lambda_{e_a}\}.
       \end{align*}
\end{cor}

\medskip

\begin{defin}
	Let $B$ be a left semi-brace such that $E$ is an ideal of $B$. We define $\Zoc_0\left(B\right):= E$, $\Zoc_1\left(B\right):=\Zoc(B)$ and, for every integer $k>1$, $\Zoc_k\left(B\right)$ is the ideal of $B$ containing $\Zoc_{k-1}\left(B\right)$ such that
		\begin{align*}
		\Zoc_k\left(B\right)/\Zoc_{k-1}\left(B\right) 
		= \Zoc\left(B/\Zoc_{k-1}\left(B\right)\right)=\Soc\left(B/\Zoc_{k-1}\left(B\right)\right).
		\end{align*}
		We call such a series the \emph{generalized socle series} of $B$.
\end{defin}

\begin{defin}
	Let $B$ be a left semi-brace. We define $\Soc_0\left(B\right):= 0$, $\Soc_1\left(B\right):=\Soc(B)$ and, for every integer $k>1$, $\Soc_k\left(B\right)$ is the ideal of $B$ containing $\Soc_{k-1}\left(B\right)$ such that
	\begin{align*}
	\Soc_k\left(B\right)/\Soc_{k-1}\left(B\right) 
	= \Soc\left(B/\Soc_{k-1}\left(B\right)\right).
	\end{align*}
	We call such a series the \emph{socle series} of $B$.
\end{defin}

\noindent Note that, for skew left braces, the previous definitions coincide with Definition $2.13$ introduced by Ced\'{o}, Smoktunowicz, and Vendramin \cite{CeSmVe19}.
\medskip

\begin{rem}\label{rem:soc-skew}
    In the case of a skew left brace $B$, using \eqref{eq:socle-skew}, we have in particular that 
\begin{align}\label{eq:socle-skew-n}
    \Soc_n\left(B\right)
    = \{a\in B \ | \ \forall \ b\in B\quad a\cdot b\in \Soc_{n-1}\left(B\right), \ [a,\, b]_+\in \Soc_{n-1}\left(B\right) \},
\end{align}
for every positive integer $n$, where $[a,\, b]_+ = - a - b + a + b$ denotes the commutator of $a$ and $b$ with respect to the sum. Note that this description is consistent with that provided by Rump in \cite[p. 161]{Ru07} for left braces where the socle series is defined by setting $\Soc_0\left(B\right) = 0$ and 
$\Soc_n\left(B\right)
    = \{a\in B \ | \ \forall \ b\in B\quad a\cdot b\in \Soc_{n-1}\left(B\right)\}$, for every positive integer $n$. Moreover, \eqref{eq:socle-skew-n} is strictly linked with \cite[Lemma 5.16]{SmVe18}.\\
 \end{rem}

\medskip

Note that $\Zoc_{n}(B) = \Soc_{n}(B) + E$, for every positive integer $n$. Indeed, by \cref{theor:semidirect-product}, $B\cong E\rtimes G$. Thus the map $\pi\colon B\rightarrow G$ defined by $\pi(a)=g_a$, for all $a\in B$, is a homomorphism of left semi-braces. 
By the assumption, $\Zoc_1(B) = \Soc(B)+E = \Soc_1(B)+E$. Now, by induction on $n$, it is easy to see that 
\begin{align}\label{ZocSoc}
	\Zoc_n(B) = \Soc_n(B) + E,
\end{align} 
for all positive integer $n$. Indeed, suppose that $\Zoc_n(B)=\Soc_n(B)+E$, for some positive integer $n$. By the induction hypothesis, we have that $$
\Zoc_{n+1}(B)/\Zoc_n(B)=\Soc(B/\Zoc_n(B))=\Soc(B/(\Soc_n(B)+E))
$$ 
and 
$$
\pi^{-1}(\Soc_{n+1}(B))/(\Soc_n(B)+E)
=(\Soc_{n+1}(B)+E)/(\Soc_n(B)+E)=
\Soc(B/(\Soc_n(B)+E)).
$$
Hence $\Zoc_{n+1}(B)=\Soc_{n+1}(B)+E$. Therefore \eqref{ZocSoc} follows by induction. 

\medskip

Now it is easy to see that for every $n>1$
	\begin{align}\label{ZocSoc2}
	\Zoc_n(B) = \{ a\in B\mid \forall b\in B,\quad a\cdot b\in \Soc_{n-1}(B),\ [g_a,g_b]_+\in \Soc_{n-1}(B)\}.
	\end{align} 
	Indeed, let $A_n= \{ a\in B\mid \forall b\in B,\quad a\cdot b\in \Soc_{n-1}(B),\ [g_a,g_b]_+\in \Soc_{n-1}(B)\}$. Let $a\in G$.
	We have that $a\circ b=a+b-g_b+\lambda_{a}(b)-g_b+g_b +\lambda_a(e_b)=a+b-g_b+a\cdot b+g_b +\lambda_a(e_b)$, $-a+b+a=-a+g_b+a=g_b+[g_b,a]_+$, for all $b\in B$. Hence $a\in A_n$ if and only if $[g_b,a]_+, -g_b+a\cdot b+g_b\in \Soc_{n-1}(B)$, and this happens if and only if $a\in \Soc_n(B)$. Since $E\subseteq A_n$, we get that $A_n=\Soc_n(B)+E=\Zoc_n(B)$.

\begin{defin}
	Let $B$ be a left semi-brace. Then, a $z$-series of $B$ is a sequence
	\begin{align*}
	B = I_0\supseteq I_1\supseteq\cdots \supseteq I_n= E
	\end{align*}
	of ideals of $B$ such that $I_{j-1}/I_{j}\subseteq \Soc\left(B/I_{j}\right)$, 
	for every $j\in\{1, \ldots, n\}$.\\
\end{defin}
\noindent Clearly, if $B$ is a skew left brace  a $z$-series of $B$ is an $s$-series of $B$
in the sense of \cite[Definition 2.11]{CeSmVe19}.
\medskip 

The following lemma is useful to prove some properties on left semi-braces admitting a $z$-series.
\begin{lemma}\label{le:E-id-I-id}
	Let $B$ be a left semi-brace such that $E$ is an ideal of $B$. Then, if $I$ is a nonempty subset of $B$, it holds that $\left(I+E\right)\cdot B\subseteq I\cdot B$.
	\begin{proof}
	This is an easy consequence of \cref{cor:E-id-1}.
	\end{proof} 
\end{lemma}
\smallskip

\begin{cor}\label{cor:right-series}
    Let $B$ be a left semi-brace such that $E$ is an ideal of $B$. Then $B^{(n)}=G^{(n)}+E$.  	
\end{cor}
\begin{proof}
	We will prove the result by induction on $n$. For $n=1$ it is clear. Suppose that the result is true for some $n\geq 1$.
	By the induction hypothesis, $B^{(n+1)}=B^{(n)}\cdot B+E=(G^{(n)}+E)\cdot B + E$. Thus by \cref{cor:E-id-1}, $B^{(n+1)}=G^{(n)}\cdot G+E = G^{(n+1)}+E$. Hence the result follows by induction.
	\end{proof}
\medskip 
	\begin{cor}\label{cor:s-series}
		Let $B$ be a left semi-brace. Then $B$ admits a $z$-series if and only if $G$ admits an $s$-series and $E$ is an ideal of $B$.  
\end{cor}
	\begin{proof}
		Suppose that $B$ admits a $z$-series
		$$B=I_0\supseteq I_1\supseteq \cdots\supseteq I_n=E.$$
		Thus $E$ is an ideal of $B$ and
		$$G\cong I_0/E\supseteq I_1/E\supseteq \cdots\supseteq I_n/E=0$$
		is an $s$-series of $B/E$, because $I_{k-1}/I_k\subseteq\Soc(B/I_k)$ and thus $(I_{k-1}/E)/(I_k/E)\subseteq \Soc((B/E)/(I_k/E))$.
		Conversely, suppose that $G$ admits an $s$-series
		$$G=J_0\supseteq J_1\supseteq \cdots\supseteq J_n=0$$
         and $E$ is an ideal of $B$. Now
         $$B=J_0+E\supseteq J_1+E\supseteq \cdots\supseteq J_n+E=E,$$ 
         and $(J_{k-1}+E)/(J_k+E)\subseteq \Soc(B/(J_k+E))$ because $B/(J_k+E)$ is naturally isomorphic to $G/J_k$ and $J_{k-1}/J_k\subseteq \Soc(G/J_k)$.     
\end{proof} 

\begin{cor}
    Let $B$ be a left semi-brace. Then $B$ admits a $z$-series if and only if $G$ has finite multipermutational level and $E$ is an ideal of $B$. 
    \begin{proof}
        The claim follows by \cref{cor:s-series}
        and \cite[Proposition 2.19]{CeSmVe19}.
    \end{proof}
\end{cor}

\bigskip

The following proposition is an easy consequence of \cref{cor:right-series} and \cite[Lemma 2.14]{CeSmVe19}.
\begin{prop}\label{prop:z-series}
	Let $B$ be a left semi-brace and $B = I_0\supseteq I_1\supseteq\cdots \supseteq I_n= E$ a $z$-series of $B$.  Then, $B^{\left(i+1\right)}\subseteq I_i$, for every $i$.
\end{prop}
\medskip

In the following proposition, we prove an extension of the characterization provided in \cite[Lemma 2.16]{CeSmVe19}. To this end, we recall that a skew left brace $B$ is said to be of \emph{nilpotent type} if the group $\left(B, +\right)$ is nilpotent.
\begin{prop}\label{prop:nilpotent-type}
Let $B$ be a left semi-brace such that $(G,+)$ is a nilpotent group. Then $B$ is right nilpotent if and only if $B$ admits a $z$-series.
   \begin{proof}
     If $B$ admits a $z$-series, then $B$ is right nilpotent by \cref{prop:z-series}.
     Conversely, suppose that $B$ is right nilpotent. We have that $E$ is an ideal of $B$ and $B/E$ is a skew left brace of nilpotent type. Since $B$ is right nilpotent,  $B/E$ is also right nilpotent. By \cite[Lemma 2.16]{CeSmVe19}, $G\cong B/E$ admits an $s$-series. Thus by \cref{cor:s-series}, the result follows. 
      \end{proof}	
\end{prop}

\begin{theor}
    Let $B$ be a left semi-brace. Then, $G$ has finite multipermutational level and $E$ is an ideal of $B$ if and only if $B$ is right nilpotent and $\left(G,+\right)$ is a nilpotent group.
    \begin{proof}
        At first let us assume that $G$ has finite multipermutational level and $E$ is an ideal of $B$. Thus, by \cite[Theorem 2.20]{CeSmVe19} it follows that the skew left brace $G$ is right nilpotent and $\left(G,+\right)$ is a nilpotent group, hence there exists a positive integer $m$ such that $G^{\left(m\right)} = 0$. Moreover, by \cref{cor:right-series}, we obtain that
        $B^{\left(m\right)} = G^{\left(m\right)}+E = E$, i.e., $B$ is right nilpotent.\\
        Conversely, suppose that $B$ is right nilpotent and $\left(G,+\right)$ is a nilpotent group. Thus, $B^{\left(m\right)} = E$, for a certain positive integer $m$ and so $E$ is an ideal of $B$. By \cref{cor:right-series} we obtain that
        $G^{\left(m\right)} = 0$, i.e., $G$ is right nilpotent. Consequently,  by \cite[Theorem 2.20]{CeSmVe19}, $G$ has finite multipermutational level. Therefore, the claim follows.
    \end{proof}
\end{theor}

\bigskip

\begin{prop}\label{prop:z-series-ch}
	Let $B$ be a left semi-brace such that $E$ is an ideal of $B$. Then, $B$ admits a $z$-series if and only if  there exists a positive integer $n$ such that $B = \Zoc_n\left(B\right)$.
	\begin{proof}
		Initially, suppose that there exists a positive integer $n$ such that $B=\Zoc_n\left(B\right)$. Thus, we have that
		\begin{align*}
		B = \Zoc_n\left(B\right)
		\supseteq\Zoc_{n-1}\left(B\right)
		\supseteq\cdots \supseteq
		\Zoc_0\left(B\right)
		= E
		\end{align*} 
		is a $z$-series of $B$.

		Conversely, assume that $B$ admits the $z$-series
		\begin{align*}
		B = I_0\supseteq I_1\supseteq\cdots \supseteq I_n = E.
		\end{align*}
		We show that $I_{n-j}\subseteq \Zoc_j\left(B\right)$ proceeding by induction on $j$. For $j=0$, $I_n = E = \Zoc_0\left(B\right)$.
		Now, suppose  $j>0$ and $I_{n-j+1}\subseteq
		\Zoc_{j-1}\left(B\right)$.
		Since $I_{n-j}/I_{n-j+1}
		\subseteq\Soc\left(B/I_{n-j-1}\right)
		= \Zoc\left(B/I_{n-j-1}\right)$, if $y\in I_{n-j}$ and $b\in B$, we obtain that
		\begin{align*}
			\left(g_y\circ b\right)\circ \left(g_y+b\right)^- \in I_{n-j+1}\subseteq \Zoc_{j-1}\left(B\right)\\
			\left(-g_y + b + g_y\right)\circ \left(b + 0\right)^- \in I_{n-j+1}\subseteq \Zoc_{j-1}\left(B\right),
		\end{align*}
		and so $I_{n-j}\subseteq \Zoc_{j}\left(B\right)$. Consequently, $B = I_0= \Zoc_{n}\left(B\right)$, hence the claim follows.
	\end{proof}
\end{prop}
\medskip

\begin{prop}
	Let $B$ be a left semi-brace with $E$ ideal of $B$ and such that $B/\Zoc\left(B\right)$ is right nilpotent. Then, $B$ is right nilpotent.
	\begin{proof}
		Observe that $\left(B/\Zoc\left(B\right)\right)^{\left(n\right)} = E\left(B/\Zoc\left(B\right)\right)$ if and only if 
		$B^{\left(n\right)}\subseteq \Zoc\left(B\right)$. Hence, applying \cref{le:E-id-I-id} to the ideal $\Soc\left(B\right)$, we get
		\begin{align*}
			B^{\left(n+1\right)}
			=B^{\left(n\right)}\cdot B + E
			\subseteq  \Zoc\left(B\right)\cdot B + E
			\subseteq \Soc\left(B\right)\cdot B + E
			= 0+E
			= E.
		\end{align*}
		Therefore, the result follows.
	\end{proof}
\end{prop}

\bigskip

Let $B$ be a left semi-brace and $b\in B$, then we set $b^{(1)}:=b$ and, for every $n > 1$, $b^{(n+1)}:= b^{(n)}\cdot b$.
\begin{defin}
	A left semi-brace $B$ is said to be \emph{right nil} if  for every $b\in B$ there exists a positive integer $m$ such that $b^{(m)} = 0$.
\end{defin}
\noindent Note that by  \cref{prop:E-id}, if a left semi-brace $B$ is right nilpotent, then $B$ is right nil.
In analogy to \cite[Question 2.34]{CeSmVe19} we pose the following question.
\begin{que}
    Let $B$ be a finite right nil left semi-brace. Is $B$ right nilpotent?
\end{que}

\smallskip

\section{Left series of a left semi-brace}

In this section, we deal with left nilpotent left semi-braces introducing the notion of left series in analogy to the right series in the previous section. Moreover, under some assumptions, we show that left nilpotency is equivalent to have that the multiplicative group is nilpotent.
\medskip

\begin{defin}
	Let $B$ be a left semi-brace. 
	We define $B^{1}:= B$ and, for every positive integer $n$,
	\begin{align*}
	B^{n + 1} 
	= B\cdot B^{n} + E .
	\end{align*}
	We call the series
	$B^{1} \supseteq B^{2} \supseteq \cdots \supseteq B^{n}\cdots $\, the \emph{left series} of $B$.
\end{defin}

Observe that if $B$ is a skew left brace then the previous definition coincides with that provided in \cite[p. 1372]{CeSmVe19}. 
Moreover, similarly to \cite[Proposition 2.2]{CeSmVe19}, the sets $B^{n}$ are left ideals of a left semi-brace $B$, as shown in the following proposition.
\begin{prop}\label{prop:leftseries}
    Let $B$ be a left semi-brace. Then $B^{n}$ is a left ideal of $B$, for every positive integer $n$.
\end{prop}

\begin{proof}
Firstly, let us observe that if $x\in B^n$ then it is clear that $x + 0\in B^{n}$.
Moreover, it holds that $B^{n+1}\cap G=B\cdot B^n$ which by definition is a subgroup of $(G,+)$.
Note that $B^{n+1}\subseteq B^n$, for all positive integer $n$, hence, if $b\in B$ and $x\in B^n$, we have
	$$
	\lambda_b(x)+0=\lambda_b(x)+\lambda_x(x^-)+x+0=b\cdot x+x+0\in B^n,
	$$ 
therefore $\lambda_b(x)= \lambda_b(x)+0+e_{\lambda_b(x)}\in B^n$.
Now, if $x,y\in B^n$, observe that
$x\circ y = x+\lambda_x\left(y\right)\in B^n$. In addition,
since $g_x=x+0\in B^n$ and $g_x^{-}\cdot g_x\in B\cdot B^n\subseteq B^n$, it follows that
	\begin{align*}
	    - \left(g_x^-\right) =
	    g_x^-\cdot g_x + g_x\in B^n.
	\end{align*}
Hence $x^{-}=\lambda_{g_x^-}(e_x)^-\circ g_x^-\in B^n$. Therefore 
$B^n$ is a subgroup of $\left(B,\circ \right)$.
Thus the result follows.
\end{proof}
\bigskip

Let us introduce the following definition.
\begin{defin}
	A left semi-brace $B$ is said to be \emph{left nilpotent} if $B^{n} = E$\, for some positive integer $n$.
\end{defin}

\noindent Clearly, if $B$ is a skew left brace, then the previous definition coincides with Definition~$2.22$ in \cite{CeSmVe19}.
\bigskip

Let $B$ be a left semi-brace and $b\in B$, then we set $b^1:=b$ and, for every positive integer $n$, $b^{n+1}:=b\cdot b^n$.
\begin{defin}
	A left semi-brace $B$ is said to be \emph{left nil} if  for all $b\in B$ there exists a positive integer $m$ such that $b^{m} = 0$.
\end{defin}
\noindent Note that by \cref{prop:be0}-$1.$, if a left brace $B$ is left nilpotent, then $B$ is left nil.
Observe that it is known that that every finite left nil left brace is left nilpotent, as shown by Smoktunowicz in \cite[Theorem 12]{Sm18-2} and \cite[Theorem 1.1]{Sm18}.
\begin{que}
    Let $B$ be a finite left nil left semi-brace. Is $B$ a left nilpotent left semi-brace?
\end{que}

Moreover, by \cite[Theorem 12]{Sm18-2}, it follows that any finite left brace $B$ is left nilpotent if and only if its multiplicative group $\left(B,\circ\right)$ is nilpotent. Recently, an analogue result has been obtained by Ced\'{o}, Smoktunowicz, and Vendramin in \cite{CeSmVe19} in the context of skew left braces, namely, any finite skew left brace $B$ with nilpotent additive group is
left nilpotent if and only if the multiplicative group of $B$ is nilpotent. Thus, we pose the following question.
\begin{que}\label{que:nilpotent}
    Let $B$ be a left nilpotent left semi-brace. Under which conditions is $(B,\circ )$ a nilpotent group?
\end{que}

\smallskip

	\begin{prop}\label{prop:leftnilp-nilp}
		Let $B$ be a left semi-brace such that $\left(G, +\right)$ is a finite nilpotent group and  $(E,\circ)$ is a nilpotent group contained in the centralizer of $G$ in $\left(B, \circ\right)$. Then, $B$ is left nilpotent if and only if the multiplicative group $\left(B, \circ\right)$ of $B$ is nilpotent.
		\begin{proof}
			Observe that if $B=E$, then the assertion is trivially satisfied.
			Thus, let us assume that $B\neq E$.
			Initially, suppose that $B$ is left nilpotent. Then, there exists an integer $n > 1$, such that $B^n=E$. It follows that $B\cdot B^{n-1} = 0$, hence $G\cdot G^{n-1} \subseteq B\cdot B^{n-1}  = 0$, i.e., the skew left brace $G$ is left nilpotent. By \cite[Theorem 4.8]{CeSmVe19}, we have that the group $\left(G, \circ\right)$ is nilpotent. Since $(E,\circ)$ is a nilpotent group contained in the centralizer of $G$ in $\left(B, \circ\right)$, we obtain that $\left(B,\circ\right)$ is the direct product of $(G,\circ)$ and $(E,\circ)$, two nilpotent subgroups. Therefore, the group $\left(B, \circ\right)$ is nilpotent.
			
			Conversely, let us suppose that $\left(B, \circ\right)$ is a nilpotent group. Hence, $\left(G, \circ\right)$ is nilpotent and by \cite[Theorem 4.8]{CeSmVe19} it follows that the skew left brace $G$ is left nilpotent, i.e., $G\cdot G^n  = 0$ for a certain positive integer $n$. 
			Moreover, since $E$ is a nilpotent group contained in the centralizer of $G$ in $\left(B, \circ\right)$, $\left(E, \circ\right)$ is a normal subgroup of $\left(B, \circ\right)$, so $E$ is an ideal of $B$. Thus, by \cref{cor:E-id-1}, it is easy to see by induction on $k>0$ that $B\cdot B^k=G\cdot G^k$. Hence, $B\cdot B^n= G\cdot G^n  = 0$, therefore $B^{n+1} = E$, i.e., $B$ is left nilpotent.
		\end{proof}
	\end{prop}

\medskip

    Note that if  $B$ is a left nilpotent left semi-brace and the group $\left(G, +\right)$ is nilpotent, then in general $\left(B, \circ\right)$ is not a nilpotent group, not even if $\left(E, \circ\right)$ is a nilpotent group.  
\begin{ex}\label{ex:semi-braceS3-2}
	Let $B$ be the left semi-brace recalled in \cref{ex:semi-braceS3}, where $\left(B,\circ\right)$ is the symmetric group $\Sym_3$ and 
	$a+b:= b\circ\varphi\left(a\right)$ with $\varphi$ the idempotent endomorphism of $\Sym_3$ defined by setting $\varphi\left(12\right) = \left(12\right)$ and $\varphi\left(123\right) = \id$. Since $G = \Soc\left(G\right)$, by \cref{cor:E-id}, we have that
	$a\cdot b = g_a\cdot g_b = 0$.
	Thus, $B^{\left(2\right)}= B^2 = E$, the groups $\left(G, +\right)$ and $\left(E, \circ\right)$ are nilpotent, but $\left(B,\circ\right)$ is not.\\
	Moreover, let us observe that clearly $E$ is not contained in the centralizer of $G$ in $\left(B, \circ\right)$.
\end{ex}

\begin{rem}\label{rem:prop-leftnilp-nilp}
		Let us observe that it is easy to check that the converse part of \cref{prop:leftnilp-nilp} is true also by assuming that $E$ is an ideal of $B$. Namely, it holds that if $B$ is a left semi-brace such that $\left(G, +\right)$ is a finite nilpotent group and $E$ is an ideal of $B$, if $\left(B, \circ\right)$ is a nilpotent group then the left semi-brace $B$ is left nilpotent.
\end{rem}

\smallskip

\section{Strongly nilpotent left semi-braces}

In this section, we define another series of left ideals for left semi-braces and investigate the relation of such a series with the left and the right series.

\medskip

Let $B$ be a left semi-brace. Let $B^{[1]}:= B$ and, for every positive integer $n$,
$$
B^{[n+1]}
= \gr\left\langle \bigcup_{i=1}^{n}B^{[i]}\cdot B^{[n+1-i]}\right\rangle_+ +E.
$$

Note that such a series is already known in the context of skew left braces \cite[p. 1376]{CeSmVe19} and 
 of left braces \cite[p. 6540]{Sm18}. In the last case, it is in particular a series of ideals of a left brace $B$.
\medskip

With similar computation in the proof of \cref{prop:leftseries}, one can prove the following result.
    \begin{prop}
        Let $B$ be a left semi-brace. Then $B^{[n]}$ is a left ideal of $B$, for every positive integer $n$.
    \end{prop}
\medskip

The following theorem illustrates the relation between the sequence of the left ideals $B^{[n]}$ and the left and right series of a left semi-brace $B$. In particular, this result is the analogue of Theorem 2.30 in \cite{CeSmVe19}. 
\begin{theor}\label{theor:genseries}
	Let $B$ be a left semi-brace. The following statements are equivalent.
	\begin{itemize}
		\item[(i)] $B^{[\alpha]}=E$ for some positive integer $\alpha$.
		\item[(ii)] $B^{(\beta)}=E$ and $B^{\gamma}=E$ for some positive integers $\beta,\gamma$.
	\end{itemize}
\begin{proof}
	Similar to the proof of \cite[Theorem 2.30]{CeSmVe19}. 
\end{proof} 	
\end{theor}

\begin{defin}
    A left semi-brace is said to be  \emph{strongly nilpotent} if $B^{[n]} = E$ for some positive integer $n$.
\end{defin}

Since any strongly nilpotent left semi-brace $B$ is right nilpotent, we have that $E$ is an ideal of $B$. Thus, it is easy to see, by induction and using \cref{cor:E-id-1}, that
\begin{align*}
    B^{[k]}=G^{[k]}+E\qquad\quad
    B^{(k)}=G^{(k)}+E\qquad\quad
    B^{k}=G^{k}+E,
\end{align*}
for every positive integer $k$.   

\smallskip

\section{Nilpotent left semi-braces}

In this section,  we introduce the notion of nilpotent left semi-brace. In particular, we show that nilpotenty implies that the multiplicative group $\left(B,\circ\right)$ is nilpotent. In addition, considering left semi-braces having the set of idempotents as an ideal, we show how nilpotency implies right nilpotency and, in some cases, the left nilpotency.

\medskip

Initially, let us give a suitable notion of annihilator of a left semi-brace that is a generalization of that introduced in the context of the skew left braces in \cite[Definition 7]{CCoSt19}.
\begin{defin}
	Let $B$ be a left semi-brace. Then, the set
	\begin{align*}
	\Ann\left(B\right)
	:= \Soc\left(B\right)\cap \Z\left(B\right) + E\cap \Z\left(B\right)
	\end{align*}
	is said to be the \emph{annihilator} of $B$, where $\Z\left(B\right)$ denotes the centre of the group $\left(B,\circ\right)$.
\end{defin}
\noindent Let us observe that 
$\Ann\left(B\right)\subseteq \left(\Soc\left(B\right) + E\right)\cap \Z\left(B\right)
= \Zoc\left(B\right)\cap \Z\left(B\right)$.
Moreover, under the assumption of $E$ ideal of $B$, it holds that
\begin{align}\label{eq:ann-puntino}
b\cdot a = a\cdot b = 0,
\end{align}
for all $a\in\Ann\left(B\right)$ and $b\in B$. Indeed, since $g_a\in \Soc\left(B\right)$, as seen in Section 2 and by \cref{cor:E-id-1}, we obtain that $a\cdot b = g_a\cdot g_b = 0$. Furthermore, by \cref{cor:E-id-1} and since $g_a\in \Soc\left(B\right)\cap\Z\left(B\right)$,  it follows
\begin{align*}
	b\cdot a
	=  g_b\cdot g_a
	= -g_b + g_b\circ g_a - g_a 
	= -g_b  -g_a + g_a\circ g_b 
	= -g_b  -g_a+ g_a + g_b 
	= 0.
\end{align*}

\begin{prop}
	Let $B$ be a left semi-brace such that $E$ is an ideal of $B$. Then, $\Ann\left(B\right)$ is an ideal of $B$.
	\begin{proof}
		We prove our assertion by using \cref{th:ideal-semi-b}.
		Thus, if $a\in \Ann\left(B\right)$, since $g_a\in\Soc\left(B\right)\cap \Z\left(B\right)$ and $e_a\in E\cap \Z\left(B\right)$, we get $a + 0 = g_a + 0\in \Ann\left(B\right)$, i.e., $\Ann\left(B\right) + 0\subseteq \Ann\left(B\right)$.
		Moreover, if $a,b\in\Ann\left(B\right)\cap G$, i.e., $a, b\in \Soc\left(B\right)\cap \Z\left(B\right)$, we obtain 
		\begin{align*}
		a  + b  
		&= a\circ\lambda_{a^-}\left(b\right)
		= a \circ b\in \Soc\left(B\right)\cap\Z\left(B\right)\\
		-a &= \lambda_{a}\left(a^-\right) = a^-\in \Soc\left(B\right)\cap\Z\left(B\right),
		\end{align*}
		i.e., $\Ann\left(B\right)\cap G$ is a subgroup of $\left(G, +\right)$.
		Furthermore, if $g\in G$ and $a\in \Soc\left(B\right)\cap\Z\left(B\right)$, 
		\begin{align*}
		g + a - g = a + g - g = a + 0 = a\in Soc\left(B\right)\cap\Z\left(B\right),
		\end{align*}
		hence $\Ann\left(B\right)\cap G$ is a normal subgroup of $\left(G, +\right)$.
		Now, if  $a\in\Ann\left(B\right)$, 
		\begin{align*}
		\lambda_g\left(a\right)
		&= \lambda_g\left(g_a\right) + \lambda_g\left(e_a\right)
		= -g + g\circ g_a -g + g\circ e_a\\
		& = -g + g_a\circ g - g + e_a\circ g&\mbox{$g_a\in \Z\left(B\right)$ and $e_a\in \Z\left(B\right)$}\\
		&= -g + g_a + g - g + g + e_a &\mbox{by \cref{cor:E-id}}\\
		&= g_a + e_a = a\in \Ann\left(B\right).
		\end{align*}
		Finally, note that $\Ann\left(B\right)$ is trivially a normal subgroup of $\left(B, \circ\right)$.
		Therefore, by \cref{th:ideal-semi-b}, the claim follows.
	\end{proof}
\end{prop}
\medskip

\begin{defin}
	Let $B$ be a left semi-brace.
	Then, for every positive integer $k$, we define
	\begin{align*}
	\Ann_k\left(B\right)
	= \Soc_{k}\left(B\right)\cap \zeta_{k}\left(B\right)+ 
	E\cap \zeta_{k}\left(B\right),
	\end{align*}
	where $\zeta_{k}\left(B\right)$ is the $k$-th center of $\left(B,\circ\right)$. 
	We call this series the \emph{annihilator series} of $B$.
\end{defin}

\begin{defin}
	Let $B$ be a left semi-brace. We say that $B$ is \emph{nilpotent} if there exists a positive integer $n$ such that $\Ann_n\left(B\right) = B$.
\end{defin}
Let us observe that, if the structure $\left(B,+,\cdot\right)$ is, in particular, a radical ring, then the previous definition coincides with the classical concept of nilpotency known for rings themselves.
\medskip

\begin{theor}\label{th:nilpotent-nilpotent}
	Let $B$ be a nilpotent left semi-brace. Then, the multiplicative group $\left(B, \circ\right)$ is a nilpotent group.
	\begin{proof}
		Since $\Ann_n\left(B\right) = B$, for a certain positive integer $n$, it follows that $G\subseteq \zeta_{n}\left(B\right)$ and $E\subseteq \zeta_{n}\left(B\right)$. 
		Now, if $b\in B$, then $b$ can be written as $b = g_b\circ\lambda_{g_b^-}\left(e_b\right)$ with $g_b\in G$ and $\lambda_{g_b^-}\left(e_b\right)\in E$, hence 
		$B = \zeta_n\left(B\right)$.
		Therefore, the group $\left(B,\circ\right)$ is nilpotent.
	\end{proof}
\end{theor}
\medskip

In particular, let us note that if $\Ann\left(B\right) = B$, then $\left(B, \circ\right)$ is an abelian group and in this case $E$ is clearly an ideal of $B$. Thus, by \eqref{eq:ann-puntino}, we obtain that $a\cdot b = 0$,  for all $a,b\in B$, hence $B^{\left(2\right)} = B^{2} = E$.

\bigskip

More in general, we can relate nilpotency with right nilpotency and, under some assumptions, with left nilpotency, as shown in the following theorem.
\begin{theor}\label{th:nilp-leftrightnilp}
	Let $B$ be a nilpotent left semi-brace such that $E$ is an ideal of $B$. Then, $B$ is right nilpotent. 
	Moreover, if $\left(G,+\right)$ is a finite nilpotent group, then $B$ is left nilpotent.
	\begin{proof}
		By the assumption, we have that $G\subseteq\Soc_{n}\left(B\right)$, for a certain positive integer $n$.  
		Let $b_1,\ldots,b_{n}\in B$. Observe that, by  \cref{cor:E-id-1}, 
		$b_1\cdot b_2 = g_{b_1}\cdot g_{b_2}$ and since $g_{b_1}\in\Soc_{n}\left(B\right)\subseteq\Soc_{n}\left(G\right)$, by equality \eqref{eq:socle-skew-n} it follows that $b_1\cdot b_2\in \Soc_{n-1}\left(G\right)$.  On applying this argument $n$ more times, we obtain that the element $\left(\left(b_1\cdot b_2\right)\cdot \ \cdots \ \cdot b_{n}\right)\in\Soc\left(G\right)$.
		Thus, if $b\in B$ and  $e\in E$, by  \cref{cor:E-id-1} it follows that
		\begin{align*}
			\left(\left(\left(b_1\cdot b_2\right)\cdot \ \cdots \ \cdot b_n\right) + e\right)\cdot b
			= \left(\left(\left(g_{b_1}\cdot g_{b_2}\right)\cdot \ \cdots \ \cdot g_{b_n}\right)\right)\cdot g_{b}
			=0. 
		\end{align*}
		Hence, $B^{\left(n+1\right)} = E$, i.e., the left semi-brace $B$ is right nilpotent.\\
		Now, suppose that $\left(G,+\right)$ is a finite nilpotent group. Thus, by \cref{th:nilpotent-nilpotent} we have that the multiplicative group $\left(B,\circ\right)$ of $B$ is a nilpotent group. Therefore, by 
		\cref{rem:prop-leftnilp-nilp}, we obtain that the left semi-brace $B$ is left nilpotent.
	\end{proof} 
\end{theor}

\smallskip

\begin{que}
	Let $B$ be a nilpotent left semi-brace. Under which conditions is $B$ left nilpotent?
	Can the condition in \cref{th:nilp-leftrightnilp} be relaxed?
\end{que}

\smallskip

\begin{theor}
    Let $B$ be a left semi-brace such that $(G,+)$ is a finite nilpotent group. Then $B$ is a strongly nilpotent left semi-brace and $(B,\circ)$ is nilpotent if and only if $B$ is a nilpotent left semi-brace and $E$ is an ideal of $B$.  
\end{theor}
\begin{proof}
If $B$ is a nilpotent left semi-brace and $E$ is an ideal of $B$, then by \cref{th:nilp-leftrightnilp}, $B$ is left and right nilpotent. By \cref{theor:genseries}, $B$ is strongly nilpotent. By \cref{th:nilpotent-nilpotent}, $(B,\circ)$ is nilpotent.

Conversely, suppose that $B$ is strongly nilpotent and $(B,\circ)$ is nilpotent. By \cref{theor:genseries}, $B$ is left and right nilpotent. In particular, $E$ is an ideal of $B$. By \cref{prop:nilpotent-type}, $B$ admits a $z$-series. By \cref{prop:z-series-ch}, there exists a positive integer $n$ such that $B=\Zoc_n(B)=\Soc_n(B)+E$. Since $(B,\circ)$ is nilpotent, there exists a positive integer $k$ such that $B=\zeta_k(B)$. Let $m=\max(n,k)$. Now we have
$$B=\Soc_m(B)+E=\Soc_m(B)\cap \zeta_m(B)+E\cap\zeta_m(B)=\Ann_m(B).$$
Hence $B$ is nilpotent and the result follows.
\end{proof}

\smallskip

\section*{Acknowledgments}
\noindent This work was partially supported by the Dipartimento di Matematica e Fisica ``Ennio De Giorgi" - Universit\`{a} del Salento. 
The first and the third author are members of GNSAGA (INdAM). The second author was partially supported by grants MINECO-FEDER MTM2017-83487-P and AGAUR 2017SGR1725 (Spain).

\smallskip

\bibliography{bibliography}

\def\cprime{$'$}
\begin{thebibliography}{10}
\expandafter\ifx\csname url\endcsname\relax
  \def\url#1{\texttt{#1}}\fi
\expandafter\ifx\csname urlprefix\endcsname\relax\def\urlprefix{URL }\fi

\bibitem{AcBo20}
E.~Acri, M.~Bonatto, Skew braces of size {$pq$}, Comm. Algebra 48~(5) (2020)
  1872--1881.
\newline\urlprefix\url{https://doi.org/10.1080/00927872.2019.1709480}

\bibitem{AcLuVe20}
E.~Acri, R.~Lutowski, L.~Vendramin, Retractability of solutions to the
  {Y}ang-{B}axter equation and {$p$}-nilpotency of skew braces, Internat. J.
  Algebra Comput. 30~(1) (2020) 91--115.
\newline\urlprefix\url{https://doi.org/10.1142/S0218196719500656}

\bibitem{Ba18}
D.~Bachiller, Solutions of the {Y}ang-{B}axter equation associated to skew left
  braces, with applications to racks, J. Knot Theory Ramifications 27~(8)
  (2018) 1850055, 36.
\newline\urlprefix\url{https://doi.org/10.1142/S0218216518500554}

\bibitem{BaCeJeOk18}
D.~Bachiller, F.~Ced\'{o}, E.~Jespers, J.~Okni\'{n}ski, Iterated matched
  products of finite braces and simplicity; new solutions of the
  {Y}ang-{B}axter equation, Trans. Amer. Math. Soc. 370~(7) (2018) 4881--4907.
\newline\urlprefix\url{https://doi.org/10.1090/tran/7180}

\bibitem{BaCeJeOk19}
D.~Bachiller, F.~Ced\'{o}, E.~Jespers, J.~Okni\'{n}ski, Asymmetric product of
  left braces and simplicity; new solutions of the {Y}ang-{B}axter equation,
  Commun. Contemp. Math. 21~(8) (2019) 1850042, 30.
\newline\urlprefix\url{https://doi.org/10.1142/S0219199718500426}

\bibitem{BaNeYa20}
V.~G. Bardakov, M.~V. Neshchadim, M.~K. Yadav, Computing skew left braces of
  small orders, Internat. J. Algebra Comput. 30~(4) (2020) 839--851.
\newline\urlprefix\url{https://doi.org/10.1142/S0218196720500216}

\bibitem{Ba72}
R.~J. Baxter, Partition function of the eight-vertex lattice model, Ann.
  Physics 70 (1972) 193--228.
\newline\urlprefix\url{https://doi.org/10.1016/0003-4916(72)90335-1}

\bibitem{CCoSt16}
F.~Catino, I.~Colazzo, P.~Stefanelli, Regular subgroups of the affine group and
  asymmetric product of radical braces, J. Algebra 455 (2016) 164--182.
\newline\urlprefix\url{https://doi.org/10.1016/j.jalgebra.2016.01.038}

\bibitem{CaCoSt17}
F.~Catino, I.~Colazzo, P.~Stefanelli, Semi-braces and the {Y}ang-{B}axter
  equation, J. Algebra 483 (2017) 163--187.
\newline\urlprefix\url{https://doi.org/10.1016/j.jalgebra.2017.03.035}

\bibitem{CCoSt19}
F.~Catino, I.~Colazzo, P.~Stefanelli, Skew left braces with non-trivial
  annihilator, J. Algebra Appl.
\newline\urlprefix\url{https://doi.org/10.1142/S0219498819500336}

\bibitem{CCoSt20}
F.~Catino, I.~Colazzo, P.~Stefanelli, The matched product of set-theoretical
  solutions of the {Y}ang-{B}axter equation, J. Pure Appl. Algebra 224~(3)
  (2020) 1173--1194.
\newline\urlprefix\url{https://doi.org/10.1016/j.jpaa.2019.07.012}

\bibitem{CaCoSt20-2}
F.~Catino, I.~Colazzo, P.~Stefanelli, The {M}atched {P}roduct of the
  {S}olutions to the {Y}ang--{B}axter {E}quation of {F}inite {O}rder, Mediterr.
  J. Math. 17, 58 (2020).
\newline\urlprefix\url{https://doi.org/10.1007/s00009-020-1483-y}

\bibitem{CCoSt20x-2}
F.~Catino, I.~Colazzo, P.~Stefanelli, Set-theoretic solutions to the
  {Y}ang-{B}axter equation and generalized semi-braces, arxiv preprint
  arXiv:2004.01606.
\newline\urlprefix\url{https://arxiv.org/abs/2004.01606}

\bibitem{CMaSt20x-2}
F.~Catino, M.~Mazzotta, P.~Stefanelli, Inverse semi-braces and the
  {Y}ang-{B}axter equation, arxiv preprint arXiv:2007.05730.
\newline\urlprefix\url{https://arxiv.org/abs/2007.05730}

\bibitem{Ce18}
F.~Ced\'{o}, Left braces: solutions of the {Y}ang-{B}axter equation, Adv. Group
  Theory Appl. 5 (2018) 33--90.
\newline\urlprefix\url{https://doi.org/10.4399/97888255161422}

\bibitem{CeJeOk14}
F.~Ced\'{o}, E.~Jespers, J.~Okni\'{n}ski, Braces and the {Y}ang-{B}axter
  equation, Comm. Math. Phys. 327~(1) (2014) 101--116.
\newline\urlprefix\url{https://doi.org/10.1007/s00220-014-1935-y}

\bibitem{CeJeOk21}
F.~Ced\'{o}, E.~Jespers, J.~Okni\'{n}ski, Every finite abelian group is a
  subgroup of the additive group of a finite simple left brace, J. Pure Appl.
  Algebra 225~(1) (2021) 106476, 10.
\newline\urlprefix\url{https://doi.org/10.1016/j.jpaa.2020.106476}

\bibitem{CeSmVe19}
F.~Ced\'{o}, A.~Smoktunowicz, L.~Vendramin, Skew left braces of nilpotent type,
  Proc. Lond. Math. Soc. (3) 118~(6) (2019) 1367--1392.
\newline\urlprefix\url{https://doi.org/10.1112/plms.12209}

\bibitem{Cr21}
T.~Crespo, Hopf {G}alois structures on field extensions of degree twice an odd
  prime square and their associated skew left braces, J. Algebra 565 (2021)
  282--308.
\newline\urlprefix\url{https://doi.org/10.1016/j.jalgebra.2020.09.005}

\bibitem{DeC19}
K.~De~Commer, Actions of skew braces and set-theoretic solutions of the
  reflection equation, Proc. Edinb. Math. Soc. (2) 62~(4) (2019) 1089--1113.
\newline\urlprefix\url{https://doi.org/10.1017/s0013091519000129}

\bibitem{Dr92}
V.~G. Drinfel\cprime~d, On some unsolved problems in quantum group theory, in:
  Quantum groups ({L}eningrad, 1990), vol. 1510 of Lecture Notes in Math.,
  Springer, Berlin, 1992, pp. 1--8.
\newline\urlprefix\url{https://doi.org/10.1007/BFb0101175}

\bibitem{ESS99}
P.~Etingof, T.~Schedler, A.~Soloviev, Set-theoretical solutions to the quantum
  {Y}ang-{B}axter equation, Duke Math. J. 100~(2) (1999) 169--209.
\newline\urlprefix\url{http://dx.doi.org/10.1215/S0012-7094-99-10007-X}

\bibitem{Ga04}
T.~Gateva-Ivanova, A combinatorial approach to the set-theoretic solutions of
  the {Y}ang-{B}axter equation, J. Math. Phys. 45~(10) (2004) 3828--3858.
\newline\urlprefix\url{https://doi.org/10.1063/1.1788848}

\bibitem{GaVa98}
T.~Gateva-Ivanova, M.~Van~den Bergh, Semigroups of {$I$}-type, J. Algebra
  206~(1) (1998) 97--112.
\newline\urlprefix\url{http://dx.doi.org/10.1006/jabr.1997.7399}

\bibitem{GuVe17}
L.~Guarnieri, L.~Vendramin, Skew braces and the {Y}ang-{B}axter equation, Math.
  Comp. 86~(307) (2017) 2519--2534.
\newline\urlprefix\url{https://doi.org/10.1090/mcom/3161}

\bibitem{JeKuVAVe19}
E.~Jespers, {\L}.~Kubat, A.~Van~Antwerpen, L.~Vendramin, Factorizations of skew
  braces, Math. Ann. 375~(3-4) (2019) 1649--1663.
\newline\urlprefix\url{https://doi.org/10.1007/s00208-019-01909-1}

\bibitem{JeKuVaVe20x}
E.~Jespers, {\L}.~Kubat, A.~Van~Antwerpen, L.~Vendramin, Radical and weight of
  skew braces and their applications to structure groups of solutions of the
  {Y}ang--{B}axter equation, arxiv preprint arXiv:2001.10967.
\newline\urlprefix\url{https://arxiv.org/abs/2001.10967}

\bibitem{JeAr19}
E.~Jespers, A.~Van~Antwerpen, Left semi-braces and solutions of the
  {Y}ang-{B}axter equation, Forum Math. 31~(1) (2019) 241--263.
\newline\urlprefix\url{https://doi.org/10.1515/forum-2018-0059}

\bibitem{Ka95}
C.~Kassel, Quantum groups, vol. 155 of Graduate Texts in Mathematics,
  Springer-Verlag, New York, 1995.
\newline\urlprefix\url{https://doi.org/10.1007/978-1-4612-0783-2}

\bibitem{KoSmVe18}
A.~Konovalov, A.~Smoktunowicz, L.~Vendramin, On skew braces and their ideals,
  Experimental Mathematics, Article in Press (2018) 1--10.
\newline\urlprefix\url{https://doi.org/10.1080/10586458.2018.1492476}

\bibitem{La20}
I.~Lau, An associative left brace is a ring, J. Algebra Appl. 19~(9) (2020)
  2050179, 6.
\newline\urlprefix\url{https://doi.org/10.1142/S0219498820501790}

\bibitem{LuYZ00}
J.-H. Lu, M.~Yan, Y.-C. Zhu, On the set-theoretical {Y}ang-{B}axter equation,
  Duke Math. J. 104~(1) (2000) 1--18.
\newline\urlprefix\url{http://dx.doi.org/10.1215/S0012-7094-00-10411-5}

\bibitem{MeBaEs19}
H.~Meng, A.~Ballester-Bolinches, R.~Esteban-Romero, Left braces and the quantum
  {Y}ang-{B}axter equation, Proc. Edinb. Math. Soc. (2) 62~(2) (2019) 595--608.
\newline\urlprefix\url{https://doi.org/10.1017/s0013091518000664}

\bibitem{Ze19}
K.~Nejabati~Zenouz, Skew braces and {H}opf-{G}alois structures of {H}eisenberg
  type, J. Algebra 524 (2019) 187--225.
\newline\urlprefix\url{https://doi.org/10.1016/j.jalgebra.2019.01.012}

\bibitem{Ru07}
W.~Rump, Braces, radical rings, and the quantum {Y}ang-{B}axter equation, J.
  Algebra 307~(1) (2007) 153--170.
\newline\urlprefix\url{https://doi.org/10.1016/j.jalgebra.2006.03.040}

\bibitem{Sm18-2}
A.~Smoktunowicz, A note on set-theoretic solutions of the {Y}ang-{B}axter
  equation, J. Algebra 500 (2018) 3--18.
\newline\urlprefix\url{https://doi.org/10.1016/j.jalgebra.2016.04.015}

\bibitem{Sm18}
A.~Smoktunowicz, On {E}ngel groups, nilpotent groups, rings, braces and the
  {Y}ang-{B}axter equation, Trans. Amer. Math. Soc. 370~(9) (2018) 6535--6564.
\newline\urlprefix\url{https://doi.org/10.1090/tran/7179}

\bibitem{SmVe18}
A.~Smoktunowicz, L.~Vendramin, On skew braces (with an appendix by {N}. {B}yott
  and {L}. {V}endramin), J. Comb. Algebra 2~(1) (2018) 47--86.
\newline\urlprefix\url{https://doi.org/10.4171/JCA/2-1-3}

\bibitem{So00}
A.~Soloviev, Non-unitary set-theoretical solutions to the quantum
  {Y}ang-{B}axter equation, Math. Res. Lett. 7~(5-6) (2000) 577--596.
\newline\urlprefix\url{https://doi.org/10.4310/MRL.2000.v7.n5.a4}

\bibitem{Ya67}
C.~N. Yang, Some exact results for the many-body problem in one dimension with
  repulsive delta-function interaction, Phys. Rev. Lett. 19 (1967) 1312--1315.
\newline\urlprefix\url{https://doi.org/10.1103/PhysRevLett.19.1312}

\end{thebibliography}

\end{document}